\documentclass{article}

\usepackage{intro}
\usepackage{amsmath}
\usepackage{amsfonts}
\usepackage{amsthm}
\usepackage{amssymb}
\usepackage[margin=1in]{geometry}
\usepackage{cite}
\usepackage{hyperref}
\usepackage{url}

\title{The inertial Jacquet--Langlands correspondence}
\author{Andrea Dotto}
\date{}

\begin{document}

\maketitle

\begin{abstract}
We give a parametrization of the simple Bernstein components of inner forms of a general linear group over a local field by two invariants constructed from type theory, and explicitly describe its behaviour under the Jacquet--Langlands correspondence. Along the way, we prove a conjecture of Broussous, S\'echerre and Stevens on preservation of endo-classes.
\end{abstract}

\tableofcontents

\section{Introduction.}

The construction of types for Bernstein components of an inner form of $\GL_n(F)$, for~$F$ a local non-archimedean field, was initiated by Bushnell and Kutzko in the split case and continued and eventually completed by Broussous, S\'echerre and Stevens in general. Meanwhile, Bushnell and Henniart provided a uniform description, for varying~$n$, of the objects which enter these constructions, relying on the basic notion of \emph{endo-class of simple characters}, and they started a programme aiming to use type theory to describe various instances of Langlands functoriality for general linear groups, such as the local Langlands correspondence, the Jacquet--Langlands correspondence, automorphic induction and base change of representations. This paper completes this programme for the Jacquet--Langlands correspondence at the level of inertial classes of representations.
\paragraph{}
Let $A=M_m(D)$ be a central simple algebra over~$F$, where~$D$ is a central division algebra of reduced degree~$d$ over~$F$. Then $G = \GL_m(D) = A^\times$ is an inner form of $H = \GL_n(F)$ for~$n = md$. Recall that the Jacquet--Langlands correspondence is a bijection
\begin{displaymath}
\JL_G : \bD(G) \to \bD(H)
\end{displaymath}
between the sets of essentially square-integrable representations (or discrete series representations) of these groups, characterized by the equality
\begin{displaymath}
(-1)^m \tr(\pi) = (-1)^n\tr(\JL_G\, \pi)
\end{displaymath}
on matching regular elliptic elements of~$G$ and~$H$. Here, $\tr(\pi)$ denotes the Harish-Chandra character of~$\pi$, identified with a function on regular semisimple elements.
\paragraph{}
The category of smooth representations of the groups~$G$ and~$H$, as for any other connected reductive group over~$F$, decomposes according to the action of the Bernstein centre. A block in the Bernstein decomposition corresponds to a component of the Bernstein variety, hence to an inertial class of supercuspidal supports, and two discrete series representations are in the same block if and only if they are unramified twists of each other. 
Since the Jacquet--Langlands correspondence commutes with twisting by characters, it yields a bijection
\begin{displaymath}
\JL_G : \fB_\mathrm{ds}(G) \to \fB_{\mathrm{ds}}(H)
\end{displaymath}
on the sets of components containing discrete series representations. By definition, these are the \emph{simple} components. 
The irreducible representations contained in a simple Bernstein component are said to form a simple inertial class.
\paragraph{}

In order to describe this map explicitly one needs a parametrization of both sides in terms of objects that can be compared to each other. 
The parametrization we use is a variant of the construction of~\cite{SecherreStevensJL}, which gets around a small ambiguity in that reference (compare~\cite[Section~9]{SecherreStevensJL}).
Unfortunately, we have found it quite involved to make it independent of all choices.
Precise definitions will be given later, and we just sketch our construction here. 

Given a simple inertial class~$\fs$ for~$\GL_m(D)$, we will attach to it two invariants that determine it completely. 
The first one, which we denote~$\cl(\fs)$, has been constructed by~\cite{BSSV} and consists of an endo-class of simple characters.
The second one is a representation of a finite general linear group, corresponding to a set of characters of the multiplicative group of a finite field through the parametrization of Green and Deligne--Lusztig. 
Its construction presents two sources of ambiguity: the action of a Galois group (denoted~$\Gal(\be_{n/\delta(\Theta_F)}/\mbf)$ in the main text) and the choice of a conjugacy class~$\kappa(\fs)$ of maximal $\beta$-extensions of endo-class~$\cl(\fs)$.

In most cases that have been treated in the literature, these ambiguities can be resolved by making an arbitrary choice. 
However, our arguments require us to work with many different groups at once, and so we need to make sure these choices are all compatible with each other.
To get around this, we will content ourselves with attaching a well-defined invariant
\[
\Lambda(\fs, \Theta_E, \kappa(\fs))
\]
to any simple inertial class~$\fs$ endowed with a conjugacy class of maximal $\beta$-extensions $\kappa(\fs)$ of endo-class~$\cl(\fs)$, and a lift~$\Theta_E$ of~$\cl(\fs)$ to its unramified parameter field.
The lift~$\Theta_E$ turns out to be exactly what is needed to remove the Galois ambiguity (this part is inspired by~\cite{BHeffective}).
With regard to~$\kappa(\fs)$, a natural choice is given by the \emph{$p$-primary} $\beta$-extensions, and although more refined choices are possible it will be enough to work with these.

This language allows us to state clearly the behaviour of these invariants under various functorial procedures.
For instance, using the modular version of type theory developed in~\cite{MStypes}, and the block decomposition of~\cite{SecherreStevensblocks}, we give a version of this construction that works over any algebraically closed coefficient field~$R$ with characteristic different from~$p$, and we determine its behaviour under reduction modulo~$\ell$ of an integral $\cbQ_\ell$-representation. We also prove a compatibility result with respect to parabolic induction.

\paragraph{}
Our main results on the Jacquet--Langlands correspondence are as follows. 
Let~$\fs_G$ and~$\fs_H$ be simple inertial classes of complex representations for the groups~$G$ and~$H$ respectively, and assume that $\fs_H = \JL_G(\fs_G)$.

\begin{thmintro}(Theorem~\ref{endo-classinvariance}.)
The equality $\cl(\fs_G) = \cl(\fs_H)$ holds.
\end{thmintro}
Since~$\cl(\fs)$ coincides with the endo-class attached to a simple inertial class in~\cite{BSSV} (see remark~\ref{compareendo-classes}) this theorem implies conjecture~9.5 in~\cite{BSSV}, the ``endo-class invariance conjecture". 

Next, we study the behaviour of the second invariant in our parametrization. Since our invariants determine a simple inertial class uniquely, these two theorems give a complete description of the Jacquet--Langlands correspondence at the level of inertial classes.

\begin{thmintro}(Theorem~\ref{comparecharactersingeneral}.)
Let~$\Theta_F = \cl(\fs_G) = \cl(\fs_H)$, and let~$\epsilon^1_G$ and~$\epsilon^1_H$ be the symplectic sign characters attached to any maximal simple character in~$G$ and~$H$ of endo-class~$\Theta_F$. 
Fix a lift~$\Theta_E \to \Theta_F$ of~$\Theta_F$ to its unramified parameter field.
Let~$\kappa_G$ and~$\kappa_H$ be the $p$-primary conjugacy classes of maximal $\beta$-extensions of endo-class~$\Theta_F$ in~$G$ and~$H$. 
Then
\begin{displaymath}
\Lambda(\fs_G, \Theta_E, \epsilon^1_G \kappa_G) = \Lambda(\fs_H, \Theta_E, \epsilon^1_H \kappa_H).
\end{displaymath}
\end{thmintro}

\paragraph{} To prove our theorems we heavily use the techniques developed in~\cite{BHJL} and~\cite{SecherreStevensJL}, which prove special cases of our results in the context of essentially tame endo-classes. 
The two invariants of a simple inertial class are constructed in section~3. 
To do this, we need to generalize some well-known properties of simple characters of~$\GL_n(F)$ to the nonsplit case, which we do in sections~2 and~3. Section~4 develops a character formula analogous to that in~\cite{BHJL}, keeping track of all choices. 
Then we prove the first of the theorems above, applying the method of~\cite{SecherreStevensJL} and a new technique to reduce to the split case. 
A comparison of character formulas then implies the supercuspidal case of the second theorem, and we deduce the general case applying a technique from~\cite[Section~8]{SecherreStevensJL}.

\paragraph{}
We end this introduction by pointing out that this paper does \emph{not} accomplish a local proof of the existence of the Jacquet--Langlands correspondence. The main problem is that the type is not directly related to the character, and even less so for non-supercuspidal discrete series representations. Via~\cite{MSreps}, our parametrization of simple inertial classes can be given independently of the Jacquet--Langlands correspondence, and one could write down a bijection of simple inertial classes by specifying its behaviour on our two invariants. 
The problem would then be to prove that the representations in matching inertial classes satisfy the character identity. 
The method we use in the paper assumes the existence of the Jacquet--Langlands transfer and manages to compute enough character values to characterize it completely, but the proof of this characterization relies upon knowing the existence of the transfer.

\paragraph{Acknowledgments.} I thank Toby Gee for suggesting the problem which led to my involvement with this subject, and Colin Bushnell, Guy Henniart, Vincent S\'echerre and Shaun Stevens for their advice and their interest in this work. The debt this paper owes to their ideas will be apparent to the reader, but this is a good place to acknowledge it explicitly. This work was supported by the Engineering and Physical Sciences Research Council [EP/L015234/1], The EPSRC Centre for Doctoral Training in Geometry and Number Theory (The London School of Geometry and Number Theory), University College London, and Imperial College London.

\subsection{Notation and conventions.}
Fix a local non-archimedean field~$F$ of residue characteristic~$p$ and an algebraic closure~$\overline{F}/F$, and write~$\mbf$ for the residue field, $\mO_F$ for the ring of integers and~$\pi_F$ for a uniformizer. Similar notation will be used for other local fields and central division algebras over them (so for instance~$\be$ is the residue field of~$E$). Write~$F_d$ for the unramified extension of~$F$ of degree~$d$ in~$\overline{F}$, and~$\mbf_d$ for the extension of~$\mbf$ of degree~$d$ in the algebraic closure of~$\mbf$ given by the residue field of the maximal unramified extension of~$F$ in~$\overline{F}$. The group of Teichm\"uller roots of unity in~$F$ is denoted~$\mu_F$, and the absolute value on~$F$ is normalized so that $|\pi_F| = |\mbf|^{-1}$. Whenever discussing simple characters, a choice of additive character~$\psi_F$ of~$F$ will be made implicitly and in such a way that whenever $E/F$ is a finite extension we have $\psi_E = \psi_F \circ \tr_{E/F}$.

Representations of a locally profinite group like~$\GL_m(D)$ will be tacitly assumed to be smooth. The coefficient field will change in the course of the paper, but will always be an algebraically closed field of characteristic different from~$p$, and we will specify it explicitly when needed. 
When working with coefficients in an extension of~$\bQ_\ell$, the notation~$\br_\ell(\pi)$ will stand for semisimplified mod~$\ell$ reduction of an integral representation~$\pi$ (which is well-defined in all cases of interest in this paper, by the results of~\cite{Vignerasrepsbook}).
When~$V$ is a finite length $\cbF_\ell[G]$-representation, we will sometimes use the notation~$[V]$ to denote the class of~$V$ in the Grothendieck group. 

Characters are not assumed to be unitary, and whenever a character~$\chi$ of a group~$G$ and a representation~$\pi$ of a subgroup $H \subseteq G$ are given, the representation $\pi \otimes \chi|_H$ will be called a \emph{twist} of~$\pi$.
When~$G$ is a $p$-adic reductive group and~$\chi$ is an unramified character of~$G$, this will be called an \emph{unramified twist} of~$\pi$.

For a central simple algebra~$A$ over~$F$ and $E/F$ a field extension in~$A$, the commutant of~$E$ in~$A$ will be denoted~$Z_A(E)$, and the centralizer and normalizer of~$E$ in~$G = A^\times$ will be denoted~$Z_G(E) = Z_A(E)^\times$ and $N_G(E)$ respectively. For $x \in G$, we write~$\ad(x)$ for the automorphism $z \mapsto xzx^{-1}$ of~$A$.

For an extension~$\bl / \bk$ of finite fields, we say that an element~$x \in \bl$ is \emph{$\bk$-regular} if it has~$[\bl : \bk]$ different conjugates under~$\Gal(\bl / \bk)$. A $\bk$-regular character of~$\bl^\times$ is defined similarly, via the right action $g : \chi \mapsto g^*\chi = \chi^g = \chi \circ g$ of $\Gal(\bl / \bk)$ on characters. In general, pullback by an automorphism~$g$ will be denoted~$g^*$. Notice that~$x \in \bl^\times$ can be~$\bk$-regular and still generate a proper subgroup of~$\bl^\times$ (consider, for instance, an extension of prime degree). For any character~$\alpha$ of~$\bl^\times$, define~$\bk[\alpha]$ as the fixed field of~$\Stab_{\Gal(\bl / \bk)}(\chi)$. It only depends on the orbit of~$\alpha$ under~$\Gal(\bl / \bk)$, which will be denoted~$[\alpha]$. 
If~$\ell$ is a prime number then the character~$\alpha$ decomposes uniquely as a product $\alpha = \alpha_{(\ell)}\alpha^{(\ell)}$ in which~$\alpha_{(\ell)}$ has order a power of~$\ell$ and~$\alpha^{(\ell)}$ has order coprime to~$\ell$; because this decomposition is unique, the orbit~$[\alpha^{(\ell)}]$ is independent of the representative~$[\alpha]$, and similarly for~$[\alpha_{(\ell)}]$. The orbit~$[\alpha^{(\ell)}]$ is the \emph{$\ell$-regular} part of~$[\alpha]$. 
We will often apply the following lemma.

\begin{lemma}\label{normdescent}
If~$\bl / \bk$ is an extension of finite fields, and~$\chi$ is a character of~$\bl^\times$, then there exists a unique $\bk$-regular character~$\chi^{\reg}$ of~$\bk[\chi]^\times$ such that~$\chi = \chi^{\reg} \circ N_{\bl/\bk[\chi]}$.
\end{lemma}
\begin{proof}
Since the norm map~$N_{\bl / \bk[\chi]}$ is surjective, for the existence part it suffices to prove that if~$N_{\bl / \bk[\chi]}(x) = 1$ then~$\chi(x) = 1$. But by Hilbert 90, $N_{\bl / \bk[\chi]}(x) = 1$ if and only if $x = \frac{g(y)}{y}$ for some~$g \in \Gal(\bl / \bk[\chi])$ and some~$y \in \bl^\times$, and then~$\chi(x) = 1$ as~$\chi$ is $\Gal(\bl / \bk[\chi] )$-stable. Uniqueness holds because~$N_{\bl / \bk[\chi]}$ is surjective, and regularity holds because the stabilizer of~$\chi$ in~$\Gal(\bl / \bk)$ is~$\Gal(\bl / \bk[\chi])$.  
\end{proof}

Throughout this article, the reduced degree of a central division algebra~$D$ over~$F$ (positive square root of the~$F$-dimension) is denoted by~$d$. Usually~$\GL_m(D)$ will denote an inner form of~$\GL_n(F)$, so that~$n = md$. The character ``absolute value of the reduced norm" is an unramified character of~$\GL_m(D)$, denoted~$\nu$. An unramified twist $\pi \otimes (\chi \circ \nu)$ will usually be written~$\chi\pi$.

\section{Maximal simple types.}\label{maximalsimpletypesI}
Let~$G = \GL_m(D)$ for a central division algebra~$D$ of reduced degree~$d$ over~$F$ (where possibly $D = F$). This is the group of $F$-points of a connected reductive group $\bG / F$, which is an inner form of~$\GL_{md, F}$ splitting over~$F_d$. Fix a central simple algebra $A$ of dimension $n^2$ over $F$ and a simple left $A$-module~$V$ such that the opposite of the endomorphism algebra $\End_A(V)$ is isomorphic to~$D$. Then~$D$ acts to the right on~$V$, and upon a choice of basis~$A$ identifies with the matrix algebra~$M_m(D)$ (passing to the opposite algebra ensures that the multiplication is as expected). In this section we recall some basic properties of the objects which go into the definition of types for cuspidal representations of~$G$. A lot of this material is standard, but we need generalizations to the non-split case of certain well-known properties of simple characters of~$\GL_n(F)$, and we could not find these in the literature.

The coefficient field for representations will be an algebraically closed field~$R$ of characteristic different from~$p$. By~\cite{SecherreStevensblocks}, an analogue of the Bernstein decomposition holds for the category of smooth $R$-representations of~$G$, and the blocks are in bijection with inertial equivalence classes of supercuspidal supports. 
We recall that a supercuspidal support consists of a pair $(L, \sigma)$, where~$L$ is the group of $F$-points of the $F$-Levi quotient of an $F$-parabolic subgroup of~$\bG$, and~$\sigma$ is a supercuspidal $R$-representation of~$L$. 
Two such pairs $(L_i, \sigma_i)$ are inertially equivalent if there exist $g \in G$ and an unramified character~$\chi$ of~$L_1$ such that $L_2 = gL_1g^{-1}$ and $\chi\sigma_1 = \ad(g)^*\sigma_2$. The block corresponding to an inertial class~$[L, \sigma]$ consists of those smooth representations all of whose irreducible subquotients have supercuspidal support in~$[L, \sigma]$ (see~\cite[Section~10.1]{SecherreStevensblocks}). As stated, this definition requires the uniqueness of supercuspidal support up to conjugacy for an irreducible representation, which is proved in~\cite[Section~6.2]{MSreps}.

The set of irreducible representations of~$G$ contained in a block is called an inertial class of representations. 
By definition, the simple inertial classes are those corresponding to inertial equivalence classes of the form $[\GL_{m/r}(D), \pi_0^{r}]$ for a divisor~$r$ of~$m$. 
An irreducible representation contained in a simple inertial class will be called a simple representation. 
Over the complex numbers, every essentially square-integrable representation is simple.

\paragraph{Lattice sequences.} To begin our discussion of type theory we consider \emph{lattice sequences} in the space~$V$, which are decreasing functions
\begin{displaymath}
\Lambda : \bZ \to \; ( \mO_D\text{-lattices in }V )
\end{displaymath}
where the right hand side is ordered by inclusion, and such that there exists a positive integer~$e$ with $\Lambda_{k+e} = \Lambda_k \fp_D$ for all~$k$. 
The number~$e$ is called the \emph{$\mO_D$-period} of the sequence; the sequence is called a \emph{chain}, or a \emph{strict sequence}, if it is strictly decreasing. 
A sequence is called \emph{uniform} (see~\cite[1/7]{Froelichprincipalorders}) if it is a chain and the dimension of $\Lambda_k/\Lambda_{k+1}$ over the residue field~$\bd$ of~$D$ is constant as~$k$ varies.

Every sequence defines a hereditary $\mO_F$-order $\fA = \fP_0(\Lambda)$ in $A$ equipped with a filtration by $\mO_F$-lattices $\fP_n(\Lambda)$, via
\begin{displaymath}
\fP_n(\Lambda) = \{ a \in A : a\Lambda_k \subseteq \Lambda_{k+n}\; \text{for all}\; k \in \bZ \}.
\end{displaymath}
The Jacobson radical~$\fP(\fA)$ of~$\fA$ then equals~$\fP_1(\Lambda)$ (see~\cite[1.2]{SecherreI}), and we write $U^n(\Lambda)$ for $1 + \fP_n(\Lambda)$. The \emph{normalizer} of a sequence is defined as
\begin{displaymath}
\fK(\Lambda) = \{ g \in A^\times : \; \text{there exists} \; n \in \bZ \text{ such that } g(\Lambda_k) = \Lambda_{k+n} \; \text{for all}  \; k \}.
\end{displaymath}
Such an integer~$n$ is then unique and denoted $v_\Lambda(g)$.
The map $g \mapsto v_{\Lambda}(g)$ is a morphism $\fK(\Lambda) \to \bZ$ whose kernel $U(\fA)$ is the unit group of~$\fA$. The unit groups of hereditary $\mO_F$-orders in~$A$ are precisely the parahoric subgroups of~$G = A^\times$. As in \cite[1.2]{SecherreI}, this set-up defines a bijection $\Lambda \mapsto \fP_0(\Lambda)$ from lattice \emph{chains} up to translation in $\bZ$ to hereditary orders in $A$. It follows that the normalizer of a lattice chain coincides with the normalizer in~$G$ of the corresponding hereditary order.

Let $E/F$ be a field extension in~$A$. An $\mO_D$-lattice sequence~$\Lambda$ in~$V$ is called \emph{$E$-pure} if $E^\times \subseteq \fK(\Lambda)$. 
This condition is equivalent to~$\Lambda$ being an~$\mO_E$-lattice sequence when~$V$ is viewed as an~$E$-vector space.
Denote by~$B = Z_A(E)$ the commutant of~$E$ in~$A$. This is a central simple algebra over~$E$, whose $E$-dimension we denote by $n'^2$. 
Then~$B$ is $E$-isomorphic to $M_{m'}(D_E)$ for some central division $E$-algebra $D_E$ of some $E$-dimension $d'^2$, and we have the identities 
\begin{displaymath}
n' = \frac{n}{[E:F]}, d' = \frac{d}{(d, [E:F])}, m'd' = n'
\end{displaymath}
as in~\cite[2.1.1]{BHJL}. 
(Here, the notation $(d, [E:F])$ stands for the highest common factor of~$d$ and~$[E:F]$.)

The need to consider general lattice sequences instead of focusing on chains, which can be done in the split case, arises from the behaviour of filtrations of hereditary orders attached to $E$-pure sequences under intersection with~$B$, $\fA \mapsto \fA \cap B$. Upon fixing a simple left $B$-module~$V_E$, one has the following result.

\begin{thm}\label{puresequences} \cite[Theorem~1.4]{SecherreStevensIV}. Given an $E$-pure lattice sequence $\Lambda$ in $V$, there exists an $\mO_{D_E}$-lattice sequence $\Gamma$ in $V_E$ such that
\begin{displaymath}
\fP_k(\Lambda) \cap B = \fP_k(\Gamma) \; \text{for all} \; k \in \bZ,
\end{displaymath}
and the normalizer $\fK(\Gamma)$ equals $\fK(\Lambda) \cap B^\times$.
Such a sequence~$\Gamma$ is unique up to translation.
\end{thm}

The sequence~$\Gamma = \tr_B\Lambda$ is called the \emph{trace} of the lattice sequence~$\Lambda$, and~$\Lambda$ is called the \emph{continuation} of~$\Gamma$. 
Notice that the theorem does not say that every $\mO_{D_E}$-lattice sequence has a continuation: this doesn't necessarily hold (see~\cite[Exemple~1.6]{SecherreStevensIV}). 
Usually, the symbol $\fB$ will denote the hereditary order~$\fA \cap B = \fP_0(\Gamma)$.

When $a, b \in \bZ$, we can rescale a lattice sequence $\Lambda$ to a new lattice sequence defined by
\begin{displaymath}
a\Lambda + b : k \mapsto \Lambda_{\left \lceil{\frac{k-b}{a}}\right \rceil}.
\end{displaymath}
The set of these sequences is called the \emph{affine class} of $\Lambda$. 
If $\Lambda = a\Lambda_0$ for a lattice chain~$\Lambda_0$, the sequence~$\Lambda$ will be called a \emph{multiple} of~$\Lambda_0$: in this case we have $\fK(\Lambda) = \fK(\Lambda_0)$, and what changes is the filtration on this group. 
The map $\Lambda \mapsto \tr_B(\Lambda)$ preserves affine classes. 
One cannot say much about the trace of an arbitrary sequence---for instance, the trace of a chain need not be a chain, see~\cite[Section~6]{BLbuilding}---but the following result on preimages holds.

\begin{pp}\label{uniformsequences}
Assume~$\Lambda$ is an $E$-pure lattice sequence in~$V$ whose trace~$\Gamma = a\Gamma_0$ is a multiple of a uniform chain~$\Gamma_0$ of $\mO_{E}$-period~$r$. Then $\Lambda$ is a multiple of a uniform chain of $\mO_D$-period~$\frac{re(E/F)}{(d, re(E/F))}$.
\end{pp}
\begin{proof}
By~\cite[Proposition~II.5.4]{BLbuilding}, if~$\Gamma$ is a multiple of a uniform chain then so is~$\Lambda$. By~\cite[Th\'eor\`eme~1.7]{SecherreStevensIV} and its proof, there exists a unique chain~$\Lambda_0$ in~$V$ whose trace is a multiple of~$\Gamma_0$, and the $\mO_D$-period of~$\Lambda_0$ is~$re(E/F)/(d, re(E/F))$ (see also~\cite[Lemma~4.18]{BSSV}). The claim now follows as~$\Lambda$ is a multiple of some chain, which must be~$\Lambda_0$, since~$\tr_B$ commutes with scaling.
\end{proof}

\paragraph{Simple characters.} 
We only discuss simple characters attached to simple strata of the form~$[\fA, \beta]$, consisting of a principal $\mO_F$-order~$\fA$ in~$A$ attached to a lattice chain~$\Lambda$ in~$V$, and an element $\beta \in A$ generating a field $E = F[\beta]$, such that $E^\times \subseteq \fK(\Lambda)$ and the condition
\begin{displaymath}
k_0(\beta, \fA) < 0
\end{displaymath} 
on the critical exponent holds (see for instance~\cite{SecherreI} for a detailed exposition of the theory). 
We follow~\cite{BHeffective} in shortening notation to $[\fA, \beta]$ for what is otherwise denoted $[\fA, -v_\Lambda(\beta), 0, \beta]$, as these are the only strata which will show up in what follows.

As in \cite[Proposition~3.42]{SecherreI} and \cite[Section~2.5]{BHJL}, there exist $\mO_F$-orders $\fh(\beta, \fA) \subseteq \fj(\beta, \fA) \subseteq \fA$ attached to a simple stratum $[\fA, \beta]$ in~$A$, with a filtration by ideals denoted $\fh^k(\beta, \fA)$, respectively $\fj^k(\beta, \fA)$. 
There are compact open subgroups $H(\beta, \fA) = \fh(\beta, \fA)^\times$ and $J(\beta, \fA) = \fj(\beta, \fA)^\times$, with filtrations by subgroups
\begin{align*}
J^k(\beta, \fA) &= J(\beta, \fA) \cap U^k(\fA) = 1+\fj^k(\beta, \fA)\\
H^k(\beta, \fA) &= H(\beta, \fA) \cap U^k(\fA) = 1+\fh^k(\beta, \fA).
\end{align*}
These groups are normalized by $J(\beta, \fA)$ and by $\fK(\fA) \cap B^\times$, $H^k$ is normal in $J^k$ and the quotients $J^k / H^k$ are finite-dimensional vector spaces over~$\bF_p$ (see \cite[Proposition~4.3]{SecherreI}). The inclusion induces isomorphisms $\fB/\fP_1(\fB) \to \fj(\beta, \fA)/\fj^1(\beta, \fA)$ and $U(\fB)/U^1(\fB) \to J(\beta, \fA)/J^1(\beta, \fA)$. 

The group $H^1(\beta, \fA)$ carries a distinguished finite set $C(\fA, \beta)$ of \emph{simple characters}, which is fundamental for the construction of types, and is defined and studied in~\cite{SecherreI} and~\cite[Section~2]{SecherreStevensIV}. 
These references treat the more general case of simple characters of positive level, which form a set $C(\fA, m, \beta)$: one has $C(\fA, \beta) = C(\fA, 0, \beta)$. The definition of~$C(\fA, \beta)$ also depends on the choice of an additive character~$\psi$ of~$F$, which is fixed throughout. Since the group $H^1(\beta, \fA)$ is a pro-$p$ group, these characters are valued in $\mu_{p^\infty}(R)$.
Hence there is a canonical bijection from the simple characters over~$\cbQ_\ell$ to those over~$\cbF_\ell$, given by reduction mod~$\ell$, whenever~$\ell \not = p$ is a prime number. 

Simple characters have various kinds of ``intertwining implies conjugacy" properties. 
In full generality, one has the following result, which can be strengthened in the split case (see \cite[Theorem~3.5.11]{BKbook} and \cite[2.6]{BHeffective}). 
In order to state it, we need the notion of an \emph{embedding} in~$A$, namely a pair $(E, \Lambda)$ where~$E$ is a field extension of~$F$ in~$A$, and~$\Lambda$ is an~$E$-pure $\mO_D$-lattice sequence in~$V$. 
Two embeddings are \emph{equivalent} if there exists~$g \in A^\times$ such that~$\Lambda_1$ and~$g\Lambda_2$ coincide up to translation, the maximal unramified extensions~$E_i^{\mathrm{ur}, d}$ of~$F$ in~$E_1$ and~$E_2$ of degree dividing~$d$ are isomorphic, and $\ad(g) E_2^{\mathrm{ur}, d} = E_1^{\mathrm{ur}, d}$.
Two simple strata $[\fA_i, \beta_i]$ have the same \emph{embedding type} if the embeddings~$(F[\beta_i], \Lambda_i)$ are equivalent, where~$\Lambda_i$ is the chain attached to~$\fA_i$.

\begin{thm}\label{intertwiningimpliesconjugacy}\cite[Theorem~1.12]{BSSV} Given two simple strata $[\fA, \beta_i]$ with the same embedding type, and two simple characters $\theta_i \in C(\fA, \beta_i)$ which intertwine in $A^\times$, let $K_i$ be the maximal unramified extension of $F$ in $F[\beta_i]$. Then there exists $u \in \fK(\fA)$ such that
\begin{enumerate}
\item $K_2 = uK_1u^{-1}$
\item $H^1(\beta_2, \fA) = uH^1(\beta_1, \fA)u^{-1}$ and $\theta_1 = \ad(u)^*\theta_2$.
\end{enumerate}
\end{thm}

\paragraph{Endo-classes.} 
Consider now all the groups $\GL_n(F)$ and their inner forms~$\GL_m(D)$ for varying~$n$, and the set of all simple characters of these groups. There is an equivalence relation on this set, called \emph{endo-equivalence}, which is defined and discussed in~\cite{BHliftingI} in the split case and~\cite{BSSV} in general. An \emph{endo-class} of simple characters over~$F$ is an equivalence class for this equivalence relation. The endo-class of a simple character~$\theta$ will be denoted~$\cl(\theta)$. Again, we identify endo-classes of simple $\cbQ_\ell$-characters and simple $\cbF_\ell$-characters.

It is important to notice that we might have two endo-equivalent simple characters~$\theta_i$ of endo-class~$\Theta_F$, defined by simple strata $[\fA_i, \beta_i]$, in which the extensions~$F[\beta_i]$ of~$F$ are not isomorphic. However, by~\cite[8.11]{BHliftingI} and \cite[Lemma~4.7]{BSSV}, they will have the same ramification index and residue class degree. The degrees $F[\beta_i]/F$ therefore also coincide. These are invariants of~$\Theta_F$, which which will be denoted $e(\Theta_F), f(\Theta_F)$ and~$\delta(\Theta_F)$ respectively.

A simple character in~$A$ is \emph{maximal} if it can be defined by a stratum $[\fA, \beta]$ such that $\fB = \fA \cap Z_A(F[\beta])$ is a maximal $\mO_{F[\beta]}$-order in~$Z_A(F[\beta])$. Such a stratum will be called a \emph{maximal simple stratum}. 
By proposition~\ref{samecharacter} below, maximality does not depend on the choice of a stratum defining~$\theta$. 
Recall (see~\cite[Definition~1.14]{BSSV}) that a simple stratum $[\fA, \beta]$ is \emph{sound} if~$\fB$ is a principal $\mO_F$-order and $\fK(\fA) \cap B^\times = \fK(\fB)$.

\begin{pp}\label{soundness}
Maximal simple strata are sound.
\end{pp}
\begin{proof}
Let $[\fA, \beta]$ be a maximal simple stratum, corresponding to a lattice chain~$\Lambda$ in~$V$. By definition, $\fA \cap B = \fB$ is a maximal order in~$B$. Consider the trace $\Gamma = \tr_B(\Lambda)$. Then $\fP_0(\Gamma) = \fB$, and it follows that the chain~$\Gamma_0$ associated to~$\Gamma$ is principal of period~1. So necessarily $\Gamma = t\Gamma_0$ for some positive integer~$t$. 
It follows that $\fK(\fB) = \fK(\Gamma_0)$ is actually equal to $\fK(\Gamma)$ (but the filtration on it may change). 
Since $\fK(\Gamma) = \fK(\Lambda) \cap B^\times$ by definition, we have $\fK(\fA) \cap B^\times = \fK(\fB)$, that is, the stratum $[\fA, \beta]$ is sound.
\end{proof}

The relation of endo-equivalence between maximal simple characters in the same group takes on a simple form: it coincides with conjugacy.
\begin{pp}\label{conjugateembeddings}
Let $[\fA_i, \beta_i]$ be maximal simple strata in the same central simple algebra~$A$ over~$F$, defining endo-equivalent maximal simple characters~$\theta_i$. 
Then~$[\fA_1, \beta_1]$ and~$[\fA_2, \beta_2]$ have the same embedding type and~$\theta_1$ and~$\theta_2$ are conjugate in~$A^\times$.
\end{pp}

\begin{proof}
Write~$B_{\beta_i} = Z_A(F[\beta_i])$, and let~$\Lambda_i$ be the lattice chains in~$V$ corresponding to the~$\fA_i$. 
By the Skolem--Noether theorem there exists~$x \in A^\times$ conjugating the maximal unramified extensions of~$F$ in~$F[\beta_i]$, as they have the same degree~$f(\Theta_F)$ over~$F$, so we can assume that they both coincide with a subfield~$E$ of~$A$. 
Because the orders~$\fB_i = \fA_i \cap B_{\beta_i}$ are maximal, there are extensions of~$F[\beta_i]$ in~$B_{\beta_i}$ which have maximal degree, are unramified, and normalize the~$\fA_i$. To see this, observe that~$\fB_i^\times$ is a maximal compact subgroup of~$B_{\beta_i}^\times$. Choose any maximal unramified extension~$L_i$ of~$F[\beta_i]$ in~$B_{\beta_i}$, so that~$\mO_{L_i}^\times$ is contained up to conjugacy in~$\fB_i^\times$. Since $L_i^\times = \pi_{F[\beta_i]}^\bZ \times \mO_{L_i}^\times$, we have $L_i^\times \subseteq \fK(\fB_i)$. By proposition~\ref{soundness}, we have $\fK(\fA_i) \cap B_{\beta_i}^\times = \fK(\fB_i)$ and so $L_i^\times \subseteq \fK(\fA_i)$.

To prove that~$[\fA_i, \beta_i]$ have the same embedding type, it is enough to prove that $\tr_{Z_A(E)}(\Lambda_i)$ are conjugate under $Z_A(E)^\times$ (up to translation), as then the same will hold for their continuations~$\Lambda_i$ by the uniqueness statement in theorem~\ref{puresequences}. 
The sequences~$\Delta_i = \tr_{Z_A(L_i)}(\Lambda_i)$ are both multiples of a chain of period~$1$, since~$Z_A(L_i) = L_i$ and so there are no other lattice sequences for~$Z_A(L_i)$.
By proposition~\ref{uniformsequences}, the sequence $\tr_{Z_A(E)}(\Lambda_i)$ is a multiple of a uniform chain~$\Gamma_i$, as its trace to~$L_i$ is~$\Delta_i$. Write~$\tr_{Z_A(E)}(\Lambda_i) = a_i\Gamma_i$. 
The $\mO_{D_E}$-period of~$\Gamma_i$ is equal to~$\frac{e(L_i/E)}{(d_{Z_A(E)}, e(L_i/E))}$, but $e(L_i/E) = e(F[\beta_i]/F)$ is independent of~$i$, hence~$\Gamma_1$ and~$\Gamma_2$ have the same $\mO_{D_E}$-period and the same~$\mO_E$-period, which we denote by~$t$. 
By the proof of~\cite[Theorem~1.7]{SecherreStevensIV}, the integer~$a_i$ then equals~$\frac{d}{(d, e(E/F)t)}=\frac{d}{(d, t)}$ and is independent of~$i$, and so the sequences $\tr_{Z_A(E)}(\Lambda_i)$ are conjugate under~$Z_A(E)^\times$ up to translation.
(By~\cite[Remark~1.8]{SecherreStevensIV} and the fact that the~$\Lambda_i$ are \emph{chains}, $a_i$ is equal to what is denoted~$\rho$ in that reference, and we are applying the formula for~$\rho$ that is given there.)

That~$\theta_1$ and~$\theta_2$ are conjugate now follows from theorem~\ref{intertwiningimpliesconjugacy}.
\end{proof}

\begin{pp}\label{samecharacter}
If $[\fA_1, \beta_1]$ and~$[\fA_2, \beta_2]$ are maximal simple strata in~$A$ defining the simple character~$\theta$, then $\fA_1 = \fA_2$ and $J^i(\beta_1, \fA_1) = J^i(\beta_2, \fA_2)$ for $i = 0, 1$.
\end{pp} 
\begin{proof}
To see that $J^i(\beta_1, \fA_1) = J^i(\beta_2, \fA_2)$ we argue as in~\cite[2.1.1]{BHeffective}. Namely, first we compute the normalizer $\bJ(\theta)$ of~$\theta$ in~$G$, as follows. We know from~\cite[Proposition~2.3]{SecherreII} that the intertwining of~$\theta$ in~$G$ is~$J(\beta, \fA)B_\beta^\times J(\beta, \fA)$, for any maximal simple stratum $[\fA, \beta]$ defining~$\theta$, and that $\fK(\fB_\beta)J(\beta, \fA)$ normalizes~$\theta$ (using that $\fK(\fB_\beta) = \fK(\fA) \cap B_\beta$ by proposition~\ref{soundness}). Now assume that $g \in B_\beta^\times$ normalizes~$\theta$. 
Then it normalizes $H^1(\beta, \fA) \cap B_\beta^\times = U^1(\fB_\beta)$ (this equality is claimed in~\cite{SecherreII} after Remarque~2.4). 
But the normalizer of $U^1(\fB_\beta)$ in~$B_\beta^\times$ equals $\fK(\fB_\beta)$, by the argument in~\cite[1.1]{BKbook}, hence $g \in \fK(\fB_\beta)$ and so the normalizer~$\bJ(\theta)$ equals~$\fK(\fB_\beta)J(\beta, \fA)$.

Since~$\fK(\fB_\beta) = D_{F[\beta]}^\times \fB_\beta^\times$, we see that $\bJ(\theta)$ has a unique maximal compact subgroup~$J_\theta$, which equals~$J(\beta, \fA)$ for any maximal simple stratum~$[\fA, \beta]$ defining~$\theta$, and~$J_\theta$ has a unique subgroup~$J^1_\theta$ that is maximal amongst its normal pro-$p$ subgroups, which is then equal to~$J^1(\beta, \fA)$. This recovers the groups $H^1, J^1$ and~$J$ intrinsically to~$\theta$.

By proposition~\ref{conjugateembeddings}, the strata $[\fA_i, \beta_i]$ have the same embedding type, hence for some $g \in G$ we have $g\fA_1 g^{-1} = \fA_2$. The characters $g^*\theta$ and~$\theta$ intertwine, hence there exists $u \in \fK(\fA_1)$ conjugating them, by theorem~\ref{intertwiningimpliesconjugacy}. So $(gu)^*\theta = \theta$, and then $gu \in \bJ(\theta)$ and conjugates~$\fA_1$ to~$\fA_2$. But $\bJ(\theta)$ normalizes~$\fA_1$, as it equals $(\fK(\fA_1) \cap B_{\beta_1})J_\theta$, and so $\fA_1 = \fA_2$, as $gu$ normalizes~$\fA_1$ and at the same time it conjugates it to~$\fA_2$.
\end{proof}

\begin{defn}
By proposition~\ref{samecharacter}, the groups $H^1(\beta, \fA), J^1(\beta, \fA)$ and~$J(\beta, \fA)$ for a simple stratum $[\fA, \beta]$ defining a maximal simple character~$\theta$ only depend on~$\theta$, and will be denoted $H^1_\theta, J^1_\theta$ and~$J_\theta$. 
\end{defn}

\paragraph{Lifting endo-classes.}
The endo-classes of~$F$ can be lifted and restricted through tamely ramified field extensions $E/F$. 
More precisely, in~\cite{BHliftingI} there is defined a restriction map
\begin{displaymath}
\Res_{E/F} : \mathcal{E}(E) \to \mathcal{E}(F)
\end{displaymath}
from the set of endo-classes of simple characters of~$E$ to those for~$F$. 
It is surjective, and its fiber over a given endo-class~$\Theta_F$ consists by definition of the set of $E$-lifts of~$\Theta_F$. 
We will need some details as to how the lifting can be performed in practice, in the unramified case.

\begin{pp}\label{interiorlifting}(\!\!\cite[Section~7]{BHliftingI} and~\cite[Sections~5 and~6]{BSSV}.)
Let~$\theta$ be a maximal simple character in~$A$ defined by the simple stratum $[\fA, \beta]$, with endo-class~$\Theta_F$. Let~$K$ be an unramified extension of~$F$ in~$A$ such that~$\beta$ commutes with~$K$ and generates a field extension of~$K$ in~$A_K = Z_A(K)$, and $K[\beta]^\times \subseteq \fK(\fA)$. Then $\theta_K = \theta|H^1_\theta \cap A_K$ is a simple character, with
\begin{align*}
H^1_{\theta_K} &= H^1_\theta \cap A_K \\
J^1_{\theta_K} &=  J^1_\theta \cap A_K \\
J_{\theta_K} &=  J_\theta \cap A_K.
\end{align*}
These groups will be denoted $H^1_K, J^1_K$ and~$J_K$ respectively. The character~$\theta_K$ is called the \emph{interior $K$-lift} of~$\theta$. Its endo-class $\Theta_K = \cl(\theta_K)$ is a $K$-lift of~$\Theta_F$.
\end{pp}

Finally, we state a compatibility of endo-classes with automorphisms of the base field. 
More generally, if $\alpha: F_1 \to F_2$ is a continuous isomorphism between local fields, it induces a pullback
\begin{equation}\label{pullbackendo-classes}
\alpha^*: \mathcal{E}(F_2) \to \mathcal{E}(F_1)
\end{equation}
on the sets of endo-classes. When a central simple algebra~$A$ over~$F_2$ is given, together with a simple character~$\theta$ in~$A^\times$, one can regard~$A$ as a central simple $F_1$-algebra via~$\alpha$, and then $\cl_{F_1}(\theta)$, the endo-class of~$\theta$ as a simple character over~$F_1$, is equal to~$\alpha^*\cl_{F_2}(\theta)$. The functoriality property
\begin{displaymath}
(\alpha_1\alpha_2)^* = \alpha_2^*\alpha_1^*
\end{displaymath}
also holds. It follows that the group of continuous automorphisms of~$F$ acts to the right on the set~$\mathcal{E}(F)$ of endo-classes of~$F$. The action will be denoted $g: \Theta_F \mapsto \Theta_F^g = g^*\Theta_F$.

\paragraph{Moving from~$H$ to~$J$.}

Let~$\theta$ be a maximal simple character in~$G$. By~\cite[Proposition~2.1]{MStypes}, there exists a unique irreducible representation~$\eta = \eta(\theta)$ of~$J^1_\theta$ which contains~$\theta$, called the \emph{Heisenberg representation} attached to~$\theta$. The dimension of~$\theta$ is a power of~$p$ and the restriction~$\eta | H^1_\theta$ is a multiple of~$\theta$, and~$\theta$ and~$\eta(\theta)$ have the same $G$-intertwining. By~\cite[Section~2.4]{MStypes}, there exists an extension of~$\eta$ to~$J_\theta$ with the same $G$-intertwining as~$\theta$ and~$\eta$, called a \emph{$\beta$-extension} or \emph{$\beta$-extension} of~$\eta$. By~\cite[Th\'eor\`eme~2.28]{SecherreII} and \cite[2.2]{MStypes} we know that the group of characters of~$\be^\times$ is transitive on the set of $\beta$-extensions of~$\eta$, by the twisting action 
\begin{equation}\label{twistbeta}
\chi: \kappa \mapsto \kappa \otimes (\chi \circ \nu_B)
\end{equation}
where~$\chi : \be^\times \to R^\times$ has been inflated to $\mO_E^\times$, and $\nu_B : \fB^\times \to \mO_E^\times$ is the reduced norm.

\begin{pp}\label{p-primarywideextension}
There exists exactly one $\beta$-extension~$\kappa$ of~$\eta$ to~$J_\theta$ such that the determinant character of~$\kappa$ has order a power of~$p$. We will refer to~$\kappa$ as a \emph{$p$-primary $\beta$-extension}.
\end{pp}
\begin{proof}
Write $E = F[\beta]$. Fix an $E$-linear isomorphism
\begin{displaymath}
\Phi: B \to M_{m'}(D')
\end{displaymath}
where~$D'$ is a central division algebra of reduced degree~$d'$ over~$E$, such that the order~$\fB$ gets mapped to~$M_{m'}(\mO_{D'})$. We then get an isomorphism $\Phi: J_\theta/J^1_\theta \to \fB^\times/U^1(\fB) \to \GL_{m'}(\bd')$, for~$\bd'$ the residue field of~$D'$, via~$\Phi$ above and the inverse of the isomorphism $\fB^\times/U^1(\fB) \to J_\theta/J^1_\theta$ induced by the inclusion.

Let~$\kappa$ be a $\beta$-extension of~$\eta$. The determinant character $\det \kappa$ has prime-to-$p$ part~$(\det \kappa)^{(p)}$ that is trivial on the pro-$p$ group $J^1_\theta$, hence $(\det \kappa)^{(p)}$ is the inflation to~$J_\theta$ of a character~$\gamma$ of~$\bd'^\times$ through the determinant of~$\GL_{m'}(\bd')$ and the isomorphism~$\Phi$. Assume that $\gamma$ is norm-inflated from $\be^\times$. Observe that $\det(\kappa \otimes (\chi \circ \nu_B)) = \det(\kappa)(\chi^{\dim \kappa} \circ \nu_B)$. Now since $\dim \kappa$ is a power of~$p$ and the character group of $\be^\times$ has order prime to~$p$, there exists a unique~$\chi$ such that $\chi^{-\dim \kappa} \circ \nu_B |_J = \det \kappa^{(p)}$, and the claim follows.

So it's enough to prove that $\gamma$ is norm-inflated from~$\be^\times$, which happens if and only if~$\gamma$ is stable under~$\Gal(\bd' / \be)$. If~$\pi_{D'}$ is a uniformizer of~$D'$, its conjugacy action on~$\fB^\times$ induces under~$\Phi$ the Frobenius automorphism on matrix entries, so it's enough to prove that the restriction of~$(\det \kappa)^{(p)}$ to~$\fB^\times$ is normalized by~$\pi_{D'}$.
This is true because~$B^\times$ intertwines~$\kappa$, hence it intertwines both~$\det \kappa$ and $(\det \kappa)^{(p)}$ with themselves.
\end{proof}

\paragraph{Maximal simple types.} Fix a maximal simple character~$\theta$ in~$G$, with corresponding Heisenberg representation~$\eta$. 
Let~$\kappa$ be a $\beta$-extension of~$\eta$ to~$J_\theta$. 
Let~$\sigma$ be a cuspidal irreducible representation of~$J_\theta/J^1_\theta$, and define~$\lambda = \sigma \otimes \kappa$. 
A pair $(J_\theta, \lambda)$ arising thus is called a \emph{maximal simple type} in~$G$. 
Over the complex numbers, these are types for the supercuspidal inertial classes of~$G$. 
The modular case is different, due to the fact that there may exist cuspidal non-supercuspidal representations, which will contain maximal simple types but will not exhaust their inertial class. 
However, the following result holds, for which see the introduction to~\cite{MStypes} and the references therein.

\begin{thm}\label{types}
Let~$\rho$ be an irreducible cuspidal representation of~$R[G]$. Then
\begin{enumerate}
\item $\rho$ contains a unique $G$-conjugacy class of maximal simple types.
\item if $(J, \lambda)$ is a maximal simple type contained in~$\rho$, then~$\lambda$ admits extensions to its normalizer~$\bJ(\lambda)$, and for precisely one such extension~$\Lambda$ the compact induction $\pi(\Lambda) = \cInd_{\bJ(\lambda)}^G\Lambda$ is isomorphic to~$\rho$.
\item two irreducible cuspidal representations~$\rho_i$ contain the same maximal simple types if and only if they are unramified twists of one another.
\end{enumerate}
\end{thm}

We will need some information on the structure of the normalizers~$\bJ(\lambda)$. 
If $(J_\theta, \lambda)$ is any maximal simple type arising from $\theta$ and~$\sigma$, and $[\fA, \beta]$ is a maximal simple stratum defining~$\theta$, by~\cite[Paragraph~3.4]{MStypes} the order~$s(\sigma)$ of the stabilizer of~$\sigma$ in~$\Gal(\bd'/ \be)$ equals the index of $F[\beta]^\times J_\theta$ in~$\bJ(\lambda)$. 
Fixing an isomorphism $B \to M_{m'}(D')$, we have $\bJ(\theta) = \fK(\fB)J_\theta = \pi_{D'}^\bZ \ltimes J_\theta$ for any uniformizer~$\pi_{D'}$ of~$D'$, and so the index of $\bJ(\lambda)$ in~$\bJ(\theta)$ equals the size~$b(\sigma)$ of the orbit of~$\sigma$ under~$\Gal(\bd'/\be)$.

\paragraph{Symplectic invariants.} 
We will work over the complex numbers until the end of this section, so we put~$R = \bC$.
Fix a maximal simple character~$\theta$ in~$G$. Then one has a well-defined map
\begin{displaymath}
J^1_\theta/H^1_\theta \times J^1_\theta/H^1_\theta \to \mu_p(\bC), (x, y) \mapsto \theta[x, y],
\end{displaymath}
where~$\mu_p(\bC)$ is the group of complex roots of unity of order~$p$ and $[x, y] = xyx^{-1}y^{-1}$. By~\cite[Proposition~2.3]{SecherreII}, this map is a symplectic form on the $\bF_p$-vector space $J^1_\theta/H^1_\theta$: it is alternating, $\bF_p$-bilinear and nondegenerate.
We are going to follow~\cite[Section~8]{BFbook} and~\cite{BHLLIII} in studying the Glauberman correspondence in this setting. 
As a consequence we will define two invariants of~$\theta$ that will appear in our character formulas, and study their behaviour under interior lifting of simple characters.
We introduce the notation
\[
N_\theta = \ker(\theta), Q_\theta = J^1_\theta/N_\theta, Z_\theta = H^1_\theta/N_\theta, V_\theta = J^1_\theta/H^1_\theta.
\]

\begin{pp}\label{Heisenbergrepresentation}\cite[8.3.3]{BFbook} 
There exists a unique irreducible representation~$\eta = \eta_\theta$ of~$Q_\theta$ which contains~$\theta$, called the \emph{Heisenberg representation} attached to~$\theta$. 
The dimension of~$\eta_\theta$ is a power of~$p$ and the restriction~$\eta_\theta | Z_\theta$ is a multiple of~$\theta$.
\end{pp}

Now choose a maximal simple stratum~$[\fA, \beta]$ defining~$\theta$.
Let~$B = Z_A(F[\beta])$ and $\fB = B \cap \fA$.
Write~$\eta$ for the corresponding Heisenberg representation of~$J^1(\beta, \fA)$, and let~$K^+$ be a maximal unramified extension of~$F[\beta]$ in~$B$, normalizing~$\fA$: this exists by the argument in the proof of proposition~\ref{conjugateembeddings}.
Let~$F \subset L \subset K^+$ be an intermediate field extension, unramified over~$F$, write~$A_L = Z_A(L)$, and recall from proposition~\ref{interiorlifting} that we have equalities
\[
J^1(\beta, \fA_L) = J^1(\beta, \fA) \cap A_L^\times \text{ and } H^1(\beta, \fA_L) = H^1(\beta, \fA) \cap A_L^\times,
\]
and~$\theta_L = \theta | H^1(\beta, \fA_L)$ is a simple character.
We will usually denote the maximal unramified extension of~$F$ in~$K^+$ by~$K$.

\begin{lemma}\label{maximallift}
$\theta_L$ is a maximal simple character.
\end{lemma}
\begin{proof}
Let~$\Lambda$ be the chain corresponding to~$\fA$.
It suffices to prove that the trace $\tr_{L[\beta]} \Lambda$ is a multiple of a uniform chain of $\mO_{D_{L[\beta]}}$-period one.
But~$\tr_{L[\beta]} \Lambda$ is the continuation of~$\tr_{K^+} \Lambda$, which is a multiple of a uniform chain of $\mO_{K^+}$-period one, since~$K^+ \subset A$ is a maximal subfield.
The claim follows by proposition~\ref{uniformsequences} since~$K^+/L[\beta]$ is unramified.
\end{proof}

The group~$\mu_K$ of roots of unity in~$K$ of order coprime to~$p$ acts by conjugation on~$Q_\theta$.
Let~$x \in \mu_K$ generate the extension~$L/F$, in the sense that~$L = F[x]$.
Then it need not be the case that~$\mu_L$ is generated by~$x$, but the argument in~\cite[Proposition~6]{BHLLIII} still implies that
\[
Q_{\theta_L} = J^1_{\theta_L}/\ker(\theta_L)
\]
coincides with the fixed point space~$Q_\theta^x$ of~$x$ acting on~$Q_\theta$.
Similarly, we have $V_{\theta_L} = V^x$.

\begin{pp}(Glauberman correspondence, \cite[Lemma~10]{BHLLIII}.)\label{pp:Glaubermancorrespondence}
View~$\eta$ as a representation of~$Q_\theta$. Then
\benum
\item there exists a unique irreducible representation $\tld \eta$ of~$\mu_K \ltimes Q_\theta$ such that
\[
\tld \eta | Q_\theta \cong \eta \text{ and } \det \tld \eta| \mu_K = 1.
\]
\item let $\Gamma \subset \mu_K$ be a subgroup. 
There is a unique irreducible representation~$\eta^\Gamma$ of~$Q_\theta^\Gamma$ such that $\eta^\Gamma| Z$ contains~$\theta$. There exists a constant $\epsilon_\theta(\Gamma) \in \{ \pm 1\}$ such that
\[
\tr \tld \eta(\gamma x) = \epsilon_\theta(\Gamma) \tr \eta^\Gamma(x) 
\]
for all $x \in Q_\theta^\Gamma$ and all generators~$\gamma$ of~$\Gamma$.
\item let~$x \in \mu_K$ and let~$L = F[x]$. Then the pullback of~$\eta^{\langle x \rangle}$ through the canonical map $J^1_{\theta_L} \to Q_{\theta_L} = Q_\theta^{\langle x \rangle}$ is isomorphic to the Heisenberg representation attached to~$\theta_L$. (This follows by the same argument as~\cite[Lemma~10(c)]{BHLLIII}.)
\eenum
\end{pp}

Given a subgroup~$\Gamma \subset \mu_K$, the sign~$\epsilon_\theta(\Gamma)$ can be described in terms of the symplectic representation~$V_\theta$ of~$\Gamma$.
More specifically, \cite[Section~3]{BHLLIII} associates to every symplectic~$\bF_p[\Gamma]$-module~$V$ a sign~$t^0_\Gamma(V)$ and a quadratic character~$t^1_\Gamma(V)$.
Then~\cite[Proposition~7]{BHLLIII} states that if~$z \in \Gamma$ is a generator of~$\Gamma$ then $\epsilon_\theta(\Gamma) = t^0_\Gamma(V_\theta) t^1_\Gamma(z, V_\theta)$.
It will be important for us that a slightly stronger result is true: by \cite[Proposition~5]{BHLLIII} we also have $\epsilon_\theta(\langle z \rangle) =  t^0_\Gamma(V_\theta)t^1_\Gamma(z, V_\theta)$ whenever $z \in \Gamma$ and the fixed point spaces $V_\theta^z$ and~$V_\theta^\Gamma$ are equal.
(We have written $\langle x \rangle$ for the group generated by~$x$, and we will sometimes write~$\epsilon_\theta(x) = \epsilon_\theta(\langle x \rangle )$. We will also follow \cite[(4.6.3)]{BHLLIII} in writing $\epsilon^i_\Gamma$ for~$t^i_\Gamma$.)

\begin{pp}\label{symplecticcomputation}
Let $x \in \mu_K$ and let~$L = F[x]$. Then we have the equality
\[
\epsilon_\theta(x) = \epsilon^0_{\mu_K}(V_{\theta_L})\epsilon^0_{\mu_K}(V_\theta)\epsilon^1_{\mu_K}(x, V_\theta).
\]
\end{pp}
\begin{proof}
We have an orthogonal decomposition
\[
V_\theta = V_\theta^{\mu_L} \perp \left ( V_\theta^{\mu_L} \right )^\perp
\]
as symplectic $\bF_p[\mu_K]$-modules. Indeed, the restriction of the symplectic form to~$V_\theta^{\mu_L}$ is nondegenerate because the order of~$\mu_L$ is prime to~$p$.

Since $V_\theta^{\mu_L} = V_\theta^x$, we have $\epsilon_\theta(x) = \epsilon^0_{\mu_L}(V_\theta)\epsilon^1_{\mu_L}(x, V_\theta)$.
Since the invariants~$t^0$ and~$t^1$ are multiplicative in orthogonal sums, we deduce that
\[
\epsilon_\theta(x) = 
\epsilon^0_{\mu_L}(V_\theta^{\mu_L}) \epsilon^1_{\mu_L}(x, V_\theta^{\mu_L})\epsilon^0_{\mu_L}(V_\theta^{\mu_L, \perp})\epsilon^1_{\mu_L}(x, V_\theta^{\mu_L, \perp}).
\]
Since~$\mu_L$ acts trivially on~$V_\theta^{\mu_L}$, we deduce that
\begin{equation}\label{secondsymplecticdecomposition}
\epsilon_\theta(x) = \epsilon^0_{\mu_L}(V_\theta^{\mu_L, \perp})\epsilon^1_{\mu_L}(x, V_\theta^{\mu_L, \perp}).
\end{equation}

On the orthogonal complement~$V_\theta^{\mu_L, \perp}$, the fixed subspaces of the groups~$\langle x \rangle, \mu_L, \mu_K$ are all trivial. Hence \cite[Proposition~5]{BHLLIII} implies that
\begin{gather*}
\epsilon_{\langle x \rangle}(V_\theta^{\mu_L, \perp}) = \epsilon^0_{\mu_L}(V_\theta^{\mu_L, \perp})\epsilon^1_{\mu_L}(x, V_\theta^{\mu_L, \perp})\\
\epsilon_{\langle x \rangle}(V_\theta^{\mu_L, \perp}) = \epsilon^0_{\mu_K}(V_\theta^{\mu_L, \perp})\epsilon^1_{\mu_K}(x, V_\theta^{\mu_L, \perp})
\end{gather*}
and so
\[
\epsilon^0_{\mu_L}(V_\theta^{\mu_L, \perp})\epsilon^1_{\mu_L}(x, V_\theta^{\mu_L, \perp}) = \epsilon^0_{\mu_K}(V_\theta^{\mu_L, \perp})\epsilon^1_{\mu_K}(x, V_\theta^{\mu_L, \perp}).
\]
Multiply both sides of this equation by~$\epsilon^0_{\mu_K}(V_{\theta_L})\epsilon^1_{\mu_K}(x, V_{\theta_L})$, and recall that~$V_{\theta_L} = V_\theta^{\mu_L}$. By~(\ref{secondsymplecticdecomposition}), the resulting equality is
\[
\epsilon^0_{\mu_K}(V_{\theta_L})\epsilon^1_{\mu_K}(x, V_{\theta_L}) \epsilon_\theta(x) = \epsilon^0_{\mu_K}(V_\theta)\epsilon^1_{\mu_K}(x, V_\theta).
\]
There remains to notice that, by construction, the character $\epsilon^1_{\mu_K}(-, V_{\theta_L})$ is the inflation to~$\mu_K$ of $\epsilon^1_{\mu_K/\mu_L}(-, V^{\mu_L})$, hence takes value~$1$ at~$x \in \mu_L$. 
\end{proof}

\section{Invariants of simple inertial classes.}\label{sec:invariants}
This section constructs the parametrization of simple inertial classes we shall use.
It consists of two invariants, the first of which is just the endo-class of simple characters contained in the inertial class.
To recall what this is, we begin with the case that~$\fs$ is a supercuspidal inertial class of $R[\GL_m(D)]$-representations.
By theorem~\ref{types} it contains a unique conjugacy class of maximal simple types, which by proposition~\ref{conjugateembeddings} determines a unique endo-class of simple characters over~$F$.
We write~$\cl(\fs)$ for this endo-class.
In the case that~$\fs$ has inertial supercuspidal support $[\GL_{m/r}(D), \pi_0^{\otimes r}]$ for a supercuspidal $R[\GL_{m/r}(D)]$-representation~$\pi_0$, we will also say that~$\fs$ is inertially equivalent to~$\fs_0^{\otimes r}$, where~$\fs_0$ is the inertial class of~$\pi_0$.
We then define~$\cl(\fs) = \cl(\fs_0)$.
We get therefore a map
\[
\cl: (\text{simple inertial classes of $R[\GL_m(D)]$-representations}) \to (\text{endo-classes of degree dividing~$n = md$})
\]
which is the first invariant we need.

To main point of this section will be to define the second invariant. 
We begin with a maximal simple character~$\theta$ of endo-class~$\Theta_F$ in $G = \GL_m(D)$. The case of level zero representations can be regarded as corresponding to maximal simple characters with trivial endo-class, which (by definition) are the trivial characters of the pro-unipotent radicals~$U^1(\fA)$ of maximal compact subgroups of~$G$. Choose a simple stratum $[\fA, \beta]$ defining~$\theta$, and let $\fB = \fA \cap B$. Since~$\theta$ is maximal, $\fB$ is a maximal order in the central simple algebra~$B$ over the field~$F[\beta]$.

Fix an $F[\beta]$-linear isomorphism
\begin{displaymath}
\Phi: B \to M_{m'}(D')
\end{displaymath}
where~$D'$ is a central division algebra of reduced degree~$d'$ over~$F[\beta]$, such that the order~$\fB$ gets mapped to~$M_{m'}(\mO_{D'})$. The inverse of the isomorphism $U(\fB)/U^1(\fB) \to J_\theta/J^1_\theta$ induced by the inclusion, together with~$\Phi$, induces an isomorphism
\begin{displaymath}
\Phi: J_\theta/J^1_\theta \to U(\fB)/U^1(\fB) \to \GL_{m'}(\bd').
\end{displaymath}

For supercuspidal classes, our second invariant should be a finite set of representations of~$\GL_{m'}(\bd')$ defined by pushing forward through the isomorphism~$\Phi$ the level zero part of a maximal simple type.
However, there is no canonical choice of~$\Phi$, and since we will need to work with many different groups at once, we need a systematic way of keeping track of choices.
To deal with this, we introduce an analogue of the notion of \emph{tame parameter field} that appears in~\cite[2.6]{BHeffective}. 
Our analogue only accounts for the unramified part of parameter fields but works uniformly across inner forms of~$\GL_n(F)$, and suffices to treat inertial classes of representations.

\subsection{Lifts to the parameter field.}\label{liftsandrigidifications}
Let $A = M_m(D)$, a central simple algebra over~$F$ with~$m \deg(D) = n$, and let~$G = A^\times = \GL_m(D)$.
Assume~$\theta$ is a maximal simple character in~$G$ and write $\Theta_F = \cl(\theta)$.

\begin{defn}
A \emph{parameter field} for~$\theta$ in~$G$ is a subfield of~$A$ of the form~$F[\beta]$ for a simple stratum $[\fA, \beta]$ defining~$\theta$. An \emph{unramified parameter field} is a subfield of~$A$ of the form $\Fbetaur$ for a parameter field~$F[\beta]$ (the maximal unramified extension of~$F$ in~$F[\beta]$).
\end{defn}

\begin{pp}\label{urparameterfields}
Let~$\theta$ be a maximal simple character in~$G$ and let $T_1, T_2$ be unramified parameter fields for~$\theta$. Then
\begin{enumerate}
\item there exists $j \in J^1_\theta$ conjugating~$T_1$ to~$T_2$
\item if~$j \in J^1_\theta$ normalizes an unramified parameter field~$T$ for~$\theta$, then it centralizes it.
\end{enumerate}
It follows that there exists exactly one isomorphism $T_1 \to T_2$ which can be realized by conjugation by elements of~$J^1_\theta$.
\end{pp}
\begin{proof}
This is very similar to~\cite[2.6 Proposition]{BHeffective}. 
Let $[\fA, \beta_i]$ be strata defining~$\theta$ with $T_i = F[\beta_i]^{\mathrm{ur}}$. For the first part, given a generator~$\zeta_1$ of~$\mu_{T_1}$, there exists some generator $\zeta_2 \in \mu_{T_2}$ and some $j_1 \in J^1_\theta$ such that $\zeta_2 = \zeta_1 j_1$. This is because the inclusion yields isomorphisms $U(\fB_{\beta_i})/U^1(\fB_{\beta_i}) \to J_\theta/J^1_\theta$ embedding~$\mu_{T_i}$ in the centre of~$J_\theta/J^1_\theta$. The centre is given by the image of $\mO_{D'_i}^\times$, hence might be larger than the image of~$\mu_{T_i}$, but it will still be a cyclic group. Since the~$\mu_{T_i}$ have the same order, as the~$T_i$ have the same degree $f(\Theta_F)$ over~$F$, they will have the same image under these maps.

By~\cite[2.6 Conjugacy Lemma]{BHeffective} we know that $\zeta_2 = \zeta_1 j_1$ is $J^1_\theta$-conjugate to some~$\zeta_3 = \zeta_1j_2$, where $j_2 \in J^1_\theta$ commutes with~$\zeta_1$. But then~$j_2 = 1$ as its order has to be both a power of~$p$ (as $j_2 \in J^1_\theta$) and prime to~$p$ (as $j_2 = \zeta_1^{-1}\zeta_3$ and the factors at the right hand side commute). So the generator~$\zeta_2$ of~$\mu_{T_2}$ is $J^1_\theta$-conjugate to the generator~$\zeta_1$ of~$\mu_{T_1}$, and the claim follows.

The second part holds as~$\mu_T$ generates~$T$ over~$F$ and embeds in~$J_\theta/J^1_\theta$, on which the conjugation action of~$J^1_\theta$ is trivial.
\end{proof}

The degree of an unramified parameter field of~$\theta$ over~$F$ equals~$f(\Theta_F)$, which is independent of the choice of~$[\fA, \beta]$ defining~$\theta$, and even of the choice of a representative~$\theta$ of~$\Theta_F$. 
Let $E = F_{f(\Theta_F)}$, the unramified extension of~$F$ in~$\overline{F}$ of degree~$f(\Theta_F)$. 
By proposition~\ref{urparameterfields}, between any two unramified parameter fields~$T_i$ for~$\theta$ there is a distinguished isomorphism $\iota_{T_1, T_2}: T_1 \to T_2$. 
Choose $F$-linear isomorphisms
\begin{displaymath}
\iota_T : E \to T
\end{displaymath}
for any unramified parameter field~$T$ for~$\theta$, such that $\iota_{T_1, T_2}\iota_{T_1} = \iota_{T_2}$ throughout. 
Denote the system of the~$\iota_T$ by~$\iota$. 

Now fix a parameter field~$F[\beta]$ for~$\theta$, and an $F[\beta]$-linear isomorphism $\Phi: B \to M_{m'}(D')$ for some central division algebra~$D'$ over~$F[\beta]$, sending~$\fB = B \cap \fA$ to~$M_{m'}(\mO_{D'})$. 
The choice of~$\iota$ yields a distinguished embedding $\be \to \bd'$, and the extension~$\bd' / \be$ then has degree~$d' = d/(d, \delta(\Theta_F))$, where $\delta(\Theta_F) = [F[\beta]:F]$, so we get a well-defined $\Gal(\be_{d'}/\be)$-orbit of $\be$-linear isomorphisms~$\bd' \isom \be_{d'}$ and $M_{m'}(\bd') \isom M_{m'}(\be_{d'})$. 
In all, the choice of~$\iota$ specifies, for every maximal simple stratum $[\fA, \beta]$ defining~$\theta$ and every $F[\beta]$-linear isomorphism $\Phi: B \to M_{m'}(D')$, an isomorphism
\begin{displaymath}
\Psi: J_\theta/J^1_\theta \to \GL_{m'}(\be_{d'}),
\end{displaymath} 
well-defined up to the action of~$\Gal(\be_{d'}/\be)$ on matrix entries.

\begin{pp}\label{compatibleconjugacy}
The conjugacy class~$\Psi(\iota)$ of this isomorphism under the natural action of $\Gal(\be_{d'}/ \be) \ltimes~\GL_{m'}(\be_{d'})$ on~$\GL_{m'}(\be_{d'})$, by inner automorphisms and Galois action on matrix entries, is independent of the choice of~$[\fA, \beta]$ and~$\Phi$, and only depends on~$\theta$ and~$\iota$.
\end{pp}

\begin{proof}
Take two maximal simple strata $[\fA, \beta_i]$ defining~$\theta$, and fix $F[\beta_i]$-linear isomorphisms $\Phi_i: B_i \to M_{m'}(D'_i)$ to central division algebras~$D'_i$ over~$F[\beta_i]$. We obtain isomorphisms
\begin{equation}\label{compatibleconjugacyequation}
J_\theta/J^1_\theta \to U(\fB_i) / U^1(\fB_i) \to \GL_{m'}(\bd'_i) \to \GL_{m'}(\be_{d'})
\end{equation}
well-defined up to Galois action on coefficients, where the first map is the inverse of the natural inclusion, the second is induced by~$\Phi_i$, and the third by an arbitrary choice of an isomorphism $\bd_i' \to \be_{d'}$ that is~$\be$-linear for the embedding~$\be \to \bd'_i$ induced by $\iota_{F[\beta_i]^\mathrm{ur}}: E \to F[\beta_i]^\mathrm{ur}$. The integers~$m_i'$ and~$d'_i$ coincide as they only depend on the endo-class of~$\theta$.

Observe that~\ref{compatibleconjugacyequation} arises from an analogous sequence
\begin{displaymath}
\fj(\beta_i, \fA) / \fj^1(\beta_i, \fA) \to \fB_i/\fP_1(\fB_i) \to M_{m'}(\bd_i') \to M_{m'}(\be_{d'})
\end{displaymath}
of $\be$-linear ring isomorphisms between $\be$-algebras, on passing to the groups of units. The equality $\fj^1(\beta_1, \fA) = \fj^1(\beta_2, \fA)$ holds since $\fj^1(\beta_i, \fA) = J^1(\beta_i, \fA) -1$. The orders $\fj(\beta_i, \fA)$ have the same group of units, since $\fj(\beta_i, \fA)^\times = J(\beta, \fA)$. The quotient $\fj(\beta_i, \fA)/\fj^1(\beta_i, \fA)$ is additively generated by its group of units (as for all matrix algebras over fields), hence $\fj(\beta_1, \fA) = \fj(\beta_2, \fA)$.

The $\be$-algebra structure on $\fj(\beta_i, \fA)/\fj^1(\beta_i, \fA)$ comes from the embedding $\iota_{F[\beta_i]^{\mathrm{ur}}}$ for $i = 1, 2$, and by construction these embeddings are conjugate by the action of~$J^1_\theta$. So these two $\be$-algebra structures coincide. 
Now the proposition follows as we have two $\be$-linear ring isomorphisms $\fj(\beta_i, \fA) / \fj^1(\beta_i, \fA) \to M_{m'}(\be_{d'})$, which therefore differ by the action of $\Gal(\be_{d'}/ \be) \ltimes \GL_{m'}(\be_{d'})$ by the Skolem--Noether theorem.
\end{proof}

We next show how the choice of a lift~$\Theta_E \to \Theta_F$ of~$\Theta_F$ to~$E$ defines such a system~$\iota$ of isomorphisms. 
Let~$[\fA, \beta_i]$ for $i = 1, 2$ be a simple stratum in~$A$ defining~$\theta$, and consider the unramified parameter field $T_i = F[\beta_i]^{\mathrm{ur}}$ of~$F$. Proposition~\ref{interiorlifting} applies, as~$\beta_i$ commutes with~$T_i$ and $T_i[\beta_i] = F[\beta_i]$ is a field with $F[\beta_i]^\times \subseteq \fK(\fA)$, and we get an interior lift~$\theta_{T_i}$. Fix compatible isomorphisms $\iota_{T_i}: E \to T_i$ as in section~\ref{liftsandrigidifications}. We get endo-classes 
\begin{displaymath}
\Theta_E^i = \iota_{T_i}^*\cl(\theta_{T_i}).
\end{displaymath}

\begin{pp}\label{sameendo-classes}
The endo-classes~$\Theta_E^1$ and~$\Theta_E^2$ are equal.
\end{pp}
\begin{proof}
Because the~$\iota_{T_i}$ are compatible, we have $\iota_{T_2} = \iota_{T_1, T_2}\iota_{T_1}$, for $\iota_{T_1, T_2}: T_1 \to T_2$ the only isomorphism induced by conjugation by elements of~$J^1_\theta$ (see proposition~\ref{urparameterfields}).
The relation 
\begin{displaymath}
\Theta_E^2 = \iota_{T_2}^*\cl(\theta_{T_2}) = \iota_{T_2}^*\iota_{T_1, T_2}^*\cl(\theta_{T_2})
\end{displaymath}
holds. Assume $\iota_{T_1, T_2}$ is induced by conjugation by~$j \in J^1_\theta$. Then 
\begin{displaymath}
\iota_{T_1, T_2}^*\cl(\theta_{T_2}) = \cl(\ad(j)^*\theta_{T_2}).
\end{displaymath}
However, $\ad(j)^*\theta_{T_2}$ is the $T_1$-lift of~$\ad(j)^*\theta = \theta$, hence $\ad(j)^*\theta_{T_2} = \theta_{T_1}$, and the claim follows.
\end{proof}

\begin{pp}\label{urparameterlift}
The group $\Gal(E/F)$ is simply transitive on the set $\Res_{E/F}^{-1}(\Theta_F)$ of $E$-lifts of~$\Theta_F$.
\end{pp}
\begin{proof}
The action has been defined at the end of section~\ref{maximalsimpletypesI}. By~\cite[1.5.1]{BHliftingIV}, $\Gal(E/F)$ is transitive on~$\Res_{E/F}^{-1}(\Theta_F)$, which is in bijection with the set of simple components of $E \otimes_F F[\beta]$ for any parameter field $F[\beta]$ for~$\theta$. But~$E$ is $F$-isomorphic to the maximal unramified extension of~$F$ in~$F[\beta]$, hence
\begin{displaymath}
E \otimes_F F[\beta] \cong \prod_{\sigma: E \to F[\beta]}F[\beta]
\end{displaymath}
and so the fiber $\Res_{E/F}^{-1}(\Theta_F)$ has as many elements as $\Gal(E/F)$.
\end{proof}

It follows that for \emph{any} unramified parameter field~$T$ for~$\theta$ we can define $\iota_T: E \to T$ to be the only $F$-linear isomorphism such that $\iota_T^*\cl(\theta_T) = \Theta_E$; by proposition~\ref{urparameterlift}, $\iota_T$ is well-defined, and by proposition~\ref{sameendo-classes} this defines a compatible system of isomorphisms. So a choice $\Theta_E \to \Theta_F$ of lift gives rise to a conjugacy class of isomorphisms
\begin{displaymath}
\Psi(\Theta_E): J_\theta/J^1_\theta \to \GL_{m'}(\be_{d'})
\end{displaymath}
for any maximal simple character~$\theta$ in~$G$ with endo-class~$\Theta_F$, by setting $\Psi(\Theta_E) = \Psi(\iota)$ for the~$\iota$ just constructed.

\subsection{Level zero maps for supercuspidal inertial classes.}\label{levelzeromaps}
Using the construction in the previous section, we attach to each supercuspidal inertial class~$\fs$ of~$G$ a conjugacy class of characters which we call the level zero part of~$\fs$.
We begin by recalling the classification of cuspidal representations of finite general linear groups.
Over~$\cbQ_\ell$ we have a bijection
\begin{displaymath}
\sigma: (\text{orbits of~$\Gal(\be_n / \be)$ on $\be$-regular characters of~$\be_n^\times$}) \to (\text{supercuspidal irreducible representations of~$\GL_n(\be)$})
\end{displaymath}
characterized by a character identity on maximal elliptic tori (see~\cite{Greencharacters} or~\cite[Section~2]{BHLLIII}). We recall that if $\be_n^\times$ is embedded in~$\GL_n(\be)$ via the left multiplication action on~$\be_n$, and $x \in \be_n^\times$ is a primitive element for the extension $\be_n / \be$, then
\begin{displaymath}
\tr \sigma[\chi](x) = (-1)^{n-1}\sum_{i=0}^{n-1}\chi(\Frob_\be^i x)
\end{displaymath}
for~$\Frob_\be$ the Frobenius element of~$\Gal(\be_n/\be)$.

Any character $\chi : \be_n^\times \to \cbQ_\ell^\times$ decomposes uniquely as a product of an $\ell$-singular part~$\chi_{(\ell)}$ and an $\ell$-regular part~$\chi^{(\ell)}$, whose orbits under~$\Gal(\be_n / \be)$ only depend on the orbit of~$\chi$. We use the mod~$\ell$ reduction map to identify the prime-to-$\ell$ roots of unity in~$\cbQ_\ell$ and~$\cbF_\ell$. Then the reduction mod~$\ell$ of~$\chi$ identifies with~$\chi^{(\ell)}$.

The reduction~$\br_\ell(\sigma[\chi])$ is irreducible and cuspidal, and only depends on~$[\chi^{(\ell)}]$ (see \cite[III.2.3 Th\'eor\`eme]{Vignerasrepsbook}). 
We denote it by~$\sigma_\ell[\chi^{(\ell)}]$. 
The map $[\chi] \mapsto \sigma_\ell[\chi^\pexp{\ell}]$ defines a bijection, from the orbits of~$\Gal(\be_n/\be)$ on the characters of~$(\be_n^\times)^{(\ell)}$ which have an $\be$-regular extension to~$\be_n^\times$, to the set of cuspidal irreducible representations of~$\GL_n(\be)$ over~$\cbF_\ell$. The representation~$\sigma_\ell[\chi^{(\ell)}]$ is supercuspidal if and only if~$[\chi^{(\ell)}]$ is itself $\be$-regular. Finally, if $\chi^{(\ell)}$ is norm-inflated from an $\be$-regular $\cbF_\ell$-character~$\chi^{(\ell), \reg}$ of~$\be_{n/a}^\times$ for some positive divisor~$a$ of~$n$, then the supercuspidal support of $\br_\ell(\sigma[\chi])$ is $\sigma_\ell[\chi^{(\ell), \reg}]^{\otimes a}$ (see~\cite[III.2.8]{Vignerasrepsbook} and \cite{MStypes} th\'eor\`eme~2.36).

The following proposition is a consequence of the classification above.

\begin{pp}\label{fixingcharacter}
Let~$R$ be an algebraically closed field of characteristic $\ell \not = p$, let $\be = \bF_q$ for~$q$ a power of~$p$, and let~$\psi$ be an $R$-character of~$\be^\times$ such that~$\psi\pi \cong \pi$ for all cuspidal irreducible $R[\GL_n(\be)]$-representations. Then~$\psi = 1$.
\end{pp}
\begin{proof}
Assume first that~$R = \cbQ_\ell$. 
By~\cite[Corollary~1.27]{DLreps}, we know that $\psi \otimes \sigma[\chi] \cong \sigma[\psi \chi]$.
Then the claim follows because if $\psi \not = 1$ there always exists an $\be$-regular character~$\chi$ of~$\be_n^\times$ with no $\Gal(\be_n/\be)$-conjugate of the form $\psi\chi$. Indeed, if $\chi^{q^i} = \psi\chi$ then $\chi^{(q-1)(q^i-1)} = 1$, and taking~$\chi$ to be a generator of the character group yields a contradiction if $0 < i \leq n-1$. Then the claim holds for any~$R$ of characteristic zero.

Now assume $R = \cbF_\ell$. 
Lift~$\psi$ to a character $\psi^+: (\be^\times)^{(\ell)} \to \cbQ_\ell^\times$.
Since $\sigma[\chi] \otimes \psi^+ \cong \sigma[\psi^+ \chi]$ for all $\be$-regular characters~$\chi$, and $\br_\ell\sigma[\chi] = \sigma_\ell[\chi^\pexp{\ell}]$, we deduce from our assumption on~$\psi$ that for all $\be$-regular $\chi: \be_n^\times \to \cbQ_\ell^\times$ we have $[\chi^{(\ell)}] = [\chi^{(\ell)}\psi^+]$. 
By duality, we find that there exists an element $x \in (\be^\times)^{(\ell)}$ such that whenever $z \in \be_n^\times$ is $\be$-regular we have $[z^{(\ell)}x] = [z^{(\ell)}]$.

Assume that $\be_n^\times$ contains an $\ell$-singular element---that is, some $\zeta \in \mu_{\ell^\infty}(\be_n)$---which is $\be$-regular. Then since $\zeta^{(\ell)} = 1$ our assumption on~$x$ implies $[x] = [1]$, hence $x = 1$.

Otherwise, since~$(\be_n^\times)_\pexp{\ell}$ is cyclic there exists a proper divisor $a | n$ such that $(\be_n^\times)_{(\ell)} = (\be_a^\times)_{(\ell)}$. 
Let~$\tau$ be a generator of $(\be_n^\times)^{(\ell)}$: then~$\tau$ is the $\ell$-regular part of some $\be$-regular element of~$\be_n^\times$, which can be chosen to be a generator of~$\be_n^\times$. 
If $x \ne 1$ then there exists a proper divisor~$b$ of~$n$ such that $(\Frob_q)^b \tau = \xi \tau$ for some $\xi \in \be^\times$, because the set of $g \in \Gal(\be_n/\be)$ with $(g\tau)\tau^{-1} \in \be^\times$ is a subgroup.

Let~$w$ be the order of $\xi \in \be^\times$, which is a divisor of~$|\be^\times| = q-1$. Then $(\Frob_q)^b(\tau^w) = (\xi \tau)^w = \tau^w$, hence $(\Frob_q)^b$ fixes the subgroup $w\cdot (\be_n^{\times})^{(\ell)}$, which has index at most~$w^{(\ell)}$ in~$(\be_n^{\times})^{(\ell)}$. 
Since $\be_n^\times \cong (\be_n^\times)^{(\ell)}\times  (\be_n^\times)_{(\ell)}$, we find a bound
\begin{displaymath}
q^n - 1 \leq w^{(\ell)}|(\be_b^\times)^{(\ell)}||(\be_a^\times)_{(\ell)}|.
\end{displaymath}
Since $w | q - 1$, we have that $w^{(\ell)}|(\be_a^\times)_{(\ell)}|$ divides $|\be_a^\times|$. Then the bound yields $q^n - 1 \leq  (q^a-1)(q^b-1)$ for certain proper divisors $a, b |n$. This is impossible even if both~$a$ and~$b$ coincide with the largest proper divisor~$d$ of~$n$, because
\begin{displaymath}
\frac{q^n - 1}{q^d - 1} = 1 + q^d + \cdots + q^{d(\frac{n}{d} -1)} > q^d -1.
\end{displaymath}

The claim then holds over~$\cbF_\ell$, and follows over an arbitrary algebraically closed field~$R$ of characteristic~$\ell$ because an irreducible $\cbF_\ell$-representation of~$\GL_n(\be)$ is absolutely irreducible, and the number of irreducible representations over~$\cbF_\ell$ and~$R$ is the same (it is the number of $\ell$-regular conjugacy classes in~$\GL_n(\be)$).
\end{proof}

\paragraph{}
Next, we need the definition and basic properties of the $\bK$-functor associated to a $\beta$-extension~$\kappa$ of a maximal simple character~$\theta$ in~$G$, as in~\cite[Section~5]{MStypes}. 
\begin{defn}
Let~$\kappa$ be a $\beta$-extension of a maximal simple character~$\theta$ in~$G$.
Define a functor from smooth representations of~$R[G]$ to representations of~$R[J_\theta/J^1_\theta]$, by
\[
\bK_\kappa: \pi \mapsto \Hom_{J^1_\theta}(\kappa|_{J^1_\theta}, \pi|_{J^1_\theta}),
\]
with the action $x: f \mapsto x f x^{-1}$ for $x \in J_\theta$.
\end{defn}

The functor~$\bK_\kappa$ is exact, and its behaviour on cuspidal irreducible representations of~$R[G]$ is recorded in the following lemma.

\begin{lemma}\cite[Lemme~5.3]{MStypes}\label{Kcuspidal}
Let~$\rho$ be a cuspidal irreducible representation of~$R[G]$. Then
\begin{enumerate}
\item if~$\rho$ does not contain~$\theta$, then~$\bK_\kappa(\rho) = 0$.
\item if~$\rho$ contains the maximal simple type $(J_\theta, \kappa \otimes \sigma)$ then
\begin{displaymath}
\bK_\kappa(\rho) \cong \sigma \oplus \sigma^\phi \oplus \cdots \oplus \sigma^{\phi^{b(\rho)-1}},
\end{displaymath} 
where~$\phi$ is a generator of~$\Gal(\be_{d'} / \be)$ and $b(\rho)$ is the size of the orbit of~$\sigma$ under the action of~$\Gal(\be_{d'} / \be)$.
\end{enumerate}
\end{lemma}

If we fix a lift~$\Theta_E \to \Theta_F$ we obtain a conjugacy class~$\Psi(\Theta_E)$ of isomorphisms $J_\theta/J^1_\theta \isom \GL_{m'}(\be_{d'})$. 
By pushforward of~$\bK_{\kappa}(\rho)$ through this conjugacy class, we can attach to every isomorphism class of~$R[G]$-representations an isomorphism class of~$R[\GL_{m'}(\be_{d'})]$-representations.
We will still denote this class by~$\bK_{\kappa}(\rho)$.
It follows from lemma~\ref{Kcuspidal} that if~$\rho$ is a cuspidal irreducible~$R[G]$-representation then~$\bK_{\kappa}(\rho)$ is the sum of an orbit of cuspidal representations of~$R[\GL_{m'}(\be_{d'})]$ under~$\Gal(\be_{d'}/\be)$.
By our discussion of representations of finite general linear groups, $\bK_{\kappa}(\rho)$ then corresponds to an orbit of~$\Gal(\be_{n/\delta(\Theta_F)} / \be)$ on the set of characters of~$\be_{n/\delta(\Theta_F)}^\times$. 
We introduce the notation $X_R(\Theta_F)$ for the set of $R^\times$-valued characters of~$\be_{n/\delta(\Theta_F)}^\times$, and $\Gamma(\Theta_F)$ for the group $\Gal(\be_{n/\delta(\Theta_F)}/ \be)$. 
We will refer to the map just constructed
\begin{displaymath}
\Lambda_{R}(-, \Theta_E, \kappa): \left ( \text{cuspidal representations of~$R[G]$ with endo-class~$\Theta_F$} \right ) \to \Gamma(\Theta_F) \backslash X_R(\Theta_F) 
\end{displaymath}
as the \emph{level zero map} attached to~$\Theta_E \to \Theta_F$ and~$\kappa$. 

\begin{pp}\label{samelevelzero}
If $\theta_1 = \ad(g)^*\theta_2$ are conjugate maximal simple characters in~$G$, and~$\kappa_1 = \ad(g)^*\kappa_2$ are $\beta$-extensions of the~$\theta_i$, then 
\begin{equation}\label{samelevelzeroII}
\Lambda_{R}(-, \Theta_E, \kappa_1) = \Lambda_{R}(-, \Theta_E, \kappa_2). 
\end{equation}
Conversely, if $\kappa_1, \kappa_2$ are $\beta$-extensions of a maximal simple character~$\theta$, and~(\ref{samelevelzeroII}) holds, then $\kappa_1 = \kappa_2$.
\end{pp}
\begin{proof}
Let~$E_1$ be an unramified parameter field for~$\theta_1$, and let~$\iota_{E_1} : E \to E_1$ be the only $F$-linear isomorphism with $\iota_{E_1}^*\cl(\theta_{1, E_1}) = \Theta_E$. Then $g E_1 g^{-1}$ is an unramified parameter field for~$\theta_2$, and we have an isomorphism $\ad(g) \circ \iota_{E_1} : E \to g E_1 g^{-1}$. Since $\theta_1 = \ad(g)^*\theta_2$, the relation $\theta_{1, E_1} = \ad(g)^*\theta_{2, g E_1 g^{-1}}$ holds on the interior lifts. Hence $(\ad(g) \circ \iota_{E_1})^*\cl \theta_{2, g E_1 g^{-1}} = \Theta_E$ and $\ad(g) \circ \iota_{E_1}$ is the isomorphism specified by~$\Theta_E$. So conjugation by~$g$ preserves the classes~$\Psi(\Theta_E)$ of isomorphisms $J_{\theta_i}/J^1_{\theta_i} \to \GL_{n/\delta(\Theta_F)}(\be)$, and since $\bK_{\kappa_1} = \ad(g)^*\bK_{\kappa_2}$ the first claim follows.

Now assume that the~$\kappa_i$ are $\beta$-extensions of~$\theta$ and~(\ref{samelevelzeroII}) holds. 
We know that the~$\kappa_i$ are twists of each other by a character~$\chi$ of~$\be^\times$. 
Then our assumption implies that~$\chi$ fixes all elements of~$\Gamma(\Theta_F) \backslash X_R(\Theta_F)$ giving rise to cuspidal $R$-representations of $\GL_{n/\delta(\Theta_F)}(\be)$, because these also give rise to cuspidal representations of~$\GL_{m'}(\be_{d'})$ (recall that $m'd' = n/\delta(\Theta_F)$). By proposition~\ref{fixingcharacter}, we have that~$\chi = 1$.
\end{proof}

Because of proposition~\ref{samelevelzero} the level zero map does not depend on the choice of~$\kappa$ within its $G$-conjugacy class.
By~\cite[Lemmas~6.1, 6.8]{MSreps}, we see that~$\Lambda_{R}(\pi, \Theta_E, \kappa)$ is supercuspidal if and only if~$\pi$ is supercuspidal. 
Finally, since an inertial class of supercuspidal representations of~$R[G]$ consists of unramified twists of a single representation, we know that the restriction of~$\Lambda_{R}(-, \Theta_E, \kappa)$ to supercuspidal representations only depends on the inertial class.
To extend the level zero map to simple inertial classes, we need to study the compatibility of $\bK$-functors with parabolic induction, which we do in the next section, following~\cite[Section~5.3]{MStypes}.

\subsection{Compatible $\beta$-extensions.}
We will develop this topic in slightly more generality than is immediately needed.
Let~$(m_1, \ldots, m_r)$ be a sequence of positive integers summing to~$m$, defining a standard Levi subgroup~$M$ of~$\GL_m(D)$. 
Let~$\Theta_F$ be an endo-class whose degree divides all the integers~$n_i = m_id$. 
Fix a $\beta$-extension~$\kappa$ of a maximal simple character~$\theta$ in~$\GL_m(D)$ of endo-class~$\Theta_F$.
Fix a lift~$\Theta_E \to \Theta_F$.
As in the proof of proposition~\ref{samelevelzero}, the effect of $\bK_\kappa$ on isomorphism classes of representations only depends on the conjugacy class of~$\kappa$ in~$G$.

\begin{defn}
For~$1 \leq i \leq r$, let~$\kappa_i$ be a conjugacy class of maximal $\beta$-extensions in~$\GL_{m_i}(D)$ of endo-class~$\Theta_F$.
We say that the sequence~$(\kappa_i)$ is compatible with~$\kappa$ if, for all irreducible representations~$\pi_i$ of~$\GL_{m_i}(D)$, there is an isomorphism
\begin{equation}\label{compatibilitycondition}
\bK_\kappa(\pi_1 \times \cdots \times \pi_r) \isom \bK_{\kappa_1}(\pi_1) \times \cdots \times \bK_{\kappa_r}(\pi_r).
\end{equation}
Here~$\prod_{i} \GL_{m_i}(D)$ is block-diagonally embedded, and the parabolic induction at the right-hand side is for the Levi subgroup $\prod_{i} \GL_{m_i'}(\be_{d'})$ of~$\GL_{m'}(\be_{d'})$.
\end{defn}

\begin{rk}
We have~$m_i' d' = m_id/\delta(\Theta_F)$, and~$d' = d/(d, \delta(\Theta_F))$, so that $m_i' = m_i(d, \delta(\Theta_F)) / \delta(\Theta_F)$ and $\sum_i m_i' = m'$.
So the parabolic induction at the right-hand side makes sense.
\end{rk}

\begin{pp}\label{parabolicinduction}
For any given choice of~$\kappa$ and~$(m_i)$, there exists a unique sequence~$(\kappa_i)$ that is compatible with~$\kappa$.
It is already determined by the existence of~(\ref{compatibilitycondition}) in the special case that $\pi_1, \ldots, \pi_r$ are cuspidal.
If~$m_i = m_j$, then $\kappa_i = \kappa_j$.
\end{pp}
\begin{proof}
Let us first prove uniqueness.
Assume we are given sequences~$(\kappa_j)$ and~$(\varkappa_j)$, compatible with~$\kappa$ in the sense that equation~(\ref{compatibilitycondition}) holds whenever the~$\pi_i$ are irreducible cuspidal representations of~$\GL_{m_i}(D)$.
There exist characters $\chi_i: \be^\times \to R^\times$ such that $\kappa_i = \chi_i\varkappa_i$.
It suffices to prove that~$\chi_i = 1$ for all~$i$.
To do so, fix~$i$ and let $[\xi]$ be a $\Gal(\be_{n_i/\delta(\Theta_F)}/\be)$-orbit of $R$-characters of~$\be_{n_i/\delta(\Theta_F)}^\times$, giving rise to a cuspidal representation of~$\GL_{m_i'd'}(\be)$.
By proposition~\ref{fixingcharacter}, it suffices to prove that~$[\xi] = [\chi_i\xi]$. 

Let~$\sigma$ be an irreducible cuspidal representation of~$\GL_{m_i'}(\be_{d'})$ attached to~$[\xi]$.
For arbitrary~$j$ let~$\mu_j$ be an irreducible cuspidal representation of $\GL_{m_j'}(\be_{d'})$, such that~$\mu_{j_1} = \mu_{j_2}$ if~$m_{j_1}' = m_{j_2}'$, and such that~$\mu_i = \sigma$.
Let~$\pi_j$ be an irreducible cuspidal representation of~$\GL_{m_j'}(D')$ containing the maximal simple type~$\kappa_j \otimes \mu_j$.

Let~$X$ be an irreducible quotient of $\mu_1 \times \cdots \times \mu_i \times \cdots \times \mu_r$.
By the compatibility assumption applied to~$\bK_\kappa(\pi_1 \times \cdots \times \pi_r)$, together with lemma~\ref{Kcuspidal}, we deduce that~$X$ is also a quotient of
\[
(\chi_1\mu_1)^{g_1} \times \cdots \times (\chi_i \mu_i)^{g_i} \times \cdots \times (\chi_r \mu_r)^{g_r}
\]
for certain $g_i \in \Gal(\be_{d'} / \be)$.
But then uniqueness of the cuspidal support for finite general linear groups implies that, whenever~$m_i' = m_j'$, we have that $\sigma$ and~$\chi_j \sigma$ are in the same~$\Gal(\be_{d'}/\be)$-orbit. 
Applying this to~$j = i$, it follows that~$[\xi] = [\chi_i\xi]$.

The existence of compatible $\beta$-extensions is established during the construction of covers of maximal simple types, and the rest of proposition~\ref{parabolicinduction} is a consequence of~\cite[Proposition~5.9]{MStypes}, although this reference does not keep track of~$\Theta_E$. 
So we review the construction briefly. 
Fix a decomposition~$D^m = D^{m_1}\oplus\cdots\oplus D^{m_r}$. 
Assume that~$\theta$ is defined by a stratum $[\Lambda_{\max}, \beta]$ such that $F[\beta]$ preserves the decomposition and~$\Lambda_{\max}$ conforms to the decomposition (see~\cite[5.3]{MStypes}). 
We obtain an embedded block-diagonal $ \prod_i A_i \cong \prod_i M_{m_i}(D)$ in~$A = M_m(D)$, for which $F[\beta]$ is diagonally embedded. 
The commutant $B=Z_A(E)$ contains the product~$B_1\times\cdots\times B_r$, where~$B_i = Z_{A_i}(E)$. 
We have furthermore lattice chains~$\Lambda_{\max, i} = \Lambda_{\max} \cap D^{m_i}$, and corresponding strata in the~$A_i$.

Now transfer~$\theta$ to maximal simple characters~$\theta_i$ defined on~$[\Lambda_{\max, i}, \beta_i]$, where~$\beta_i$ is the projection of~$\beta$ to~$B_i$. 
We are going to see that~$\kappa$ determines a $\beta$-extension of each of these~$\theta_i$. 
For this, we need to fix an $F[\beta]$-linear isomorphism $\Phi : B \to M_{m}(F[\beta])$ identifying $U^0(\Lambda_{\max}) \cap B$ with $M_{m}(\mO_{F[\beta]})$, and a simple stratum~$[\Lambda, \beta]$ in~$A$ satisfying conditions~(1) and~(2) in~\cite[Section~5.3]{MStypes}.

Transfer~$\theta$ to a simple character~$\theta_{\Lambda}$ defined on~$[\Lambda, \beta]$ (it will not be maximal) and let~$\kappa_\Lambda$ be the transfer of~$\kappa$ to a $\beta$-extension of~$\theta_\Lambda$. 
Then let~$N$ be the upper-triangular unipotent group defined by $(m_1, \ldots, m_r)$, and take the invariants of~$\kappa_\Lambda$ under $J(\beta, \Lambda) \cap N$: this is a representation of~$J(\beta, \Lambda) \cap M$ which by~\cite[Proposition~6.6]{SecherreStevensVI} decomposes as a tensor product of $\beta$-extensions~$\kappa_i$ of the~$\theta_i$. 
If~$m_i = m_j$ then the same reference shows that $\kappa_i \cong \kappa_j$: this implies the second claim of the proposition once we prove that~$(\kappa_i)$ is compatible with~$\kappa$.

To do so we apply~\cite[Proposition~5.9]{MStypes}, which says that we have an isomorphism
\begin{displaymath}
\bK_{\kappa}(\pi_1 \times \cdots \times \pi_r) \isom \bK_{\kappa_1}(\pi_1) \times \cdots \times \bK_{\kappa_r}(\pi_r)
\end{displaymath}
where both sides are regarded as representations of~$J_\theta/J^1_\theta$, and the parabolic induction refers to $\prod_{i} J_{\theta_i} / J^1_{\theta_i}$, identified with a Levi subgroup of~$J_\theta/J^1_\theta$. 
So there remains to prove that any isomorphism $J_\theta/J^1_{\theta} \to \GL_{m'}(\be_{d'})$ in the class~$\Psi(\Theta_E)$ restricts to an isomorphism $\prod_i J_{\theta_i} / J^1_{\theta_i} \to \prod_i \GL_{m_i'}(\be_{d'})$ which is in the class~$\Psi(\Theta_E)$ on each factor. 

To prove this, observe that the lift~$\Theta_E$ defines an $F$-linear embedding $\iota: E \to F[\beta]$, characterized by the equality $\cl(\iota^*\theta_{\Fbetaur}) = \Theta_E$. 
By construction, the image of~$F[\beta]$ under the projection to~$B_i$ is a parameter field for~$\theta_i$.
We need to prove that the equality $\cl(\iota^*(\theta_{i, \Fbetaur})) = \Theta_E$ holds.
Recall that~$\theta_i$ is the transfer of~$\theta$ to~$\fA_{\max, i}$. We know by assumption that the interior lift~$\theta_{\Fbetaur}$ has endo-class~$\Theta_E$ under~$\iota$, and the content of the lemma is that the same is true for these transfers. This follows from the compatibility between interior lifts and transfer maps, for which see for instance~\cite[Theorem~6.7]{BSSV}.

This concludes the proof of the proposition.
\end{proof}

Whenever we have a sequence $(\kappa_i)$ as above with the same endo-class as~$\kappa$, it makes sense to ask whether they are compatible. 
As remarked above, a necessary condition is that $\kappa_i \cong \kappa_j$ if~$i= j$. 
We will be interested in the case that all the~$m_i$ are equal to~$m/r$ for some positive divisor~$r$ of~$m$, and we can make the following definition.

\begin{defn}
We say that two conjugacy classes of maximal $\beta$-extensions~$\kappa_{0}$ in~$\GL_{m/r}(D)$ and~$\kappa$ in~$\GL_m(D)$, of endo-class~$\Theta_F$, are \emph{compatible} if $\kappa$ is compatible with~$(\kappa_{0}, \ldots, \kappa_{0})$.
\end{defn}

\begin{rk}
Every element of a compatible pair~$(\kappa_{0}, \kappa)$ determines the other uniquely: this follows from proposition~\ref{parabolicinduction} and the fact that if~$(\kappa_{0}, \kappa)$ are compatible then so are~$(\chi\kappa_{0}, \chi\kappa)$ for all characters~$\chi$ of~$\be^\times$, because $\bK_\kappa \cong \chi\bK_{\chi\kappa}$.
\end{rk}

We end this section by recording the following transitivity result with respect to refining the decomposition~$(m_i)$.

\begin{pp}\label{refinement} 
Let~$(m_i)_{i \in I}$ be a sequence summing to~$m$. Assume that for all~$i$ we have a sequence $(m_{ij})_{j \in J_i}$ of positive integers summing to~$m_i$, such that~$\delta(\Theta_F)$ divides every~$n_{ij} = m_{ij}d$. 
Let~$\kappa$ be a $\beta$-extension in~$\GL_m(D)$ of endo-class~$\Theta_F$. Let~$(\kappa_i)_{i \in I}$ be a sequence of $\beta$-extensions compatible with~$\kappa$, and for all~$i$ let $(\kappa_{ij})_{j \in J_i}$ be a sequence compatible with~$\kappa_i$. Then $(\kappa_{ij})$ is compatible with~$\kappa$.
\end{pp}
\begin{proof}
Fix irreducible cuspidal representations~$\rho_{ij}$ of~$\GL_{m_{ij}}(D)$, and form $\bK$-functors with respect to a fixed lift~$\Theta_E$.
We can twist the~$\rho_{ij}$ by unramified characters, and assume that no two of them are linked.
By~\cite[Th\'eor\`eme~7.24]{MSreps}, this implies that the parabolic inductions $\times_{j \in J_i} \rho_{ij}$ are irreducible.
So by proposition~\ref{parabolicinduction} we deduce that there are isomorphisms
\begin{equation}\label{partialcompatibility}
\bK_{\kappa}(\times_{i, j}\rho_{ij})  \isom \times_{i \in I} \bK_{\kappa_i}(\times_{j \in J_i} \rho_{ij}) \isom \times_{ij}\bK_{\kappa_{ij}}(\rho_{ij}).
\end{equation}
Since the~$\bK$-functors only depend on the restriction of a representation to a compact open subgroup, they are invariant under unramified twists. 
Hence~(\ref{partialcompatibility}) is actually true for all cuspidal irreducible representations~$\rho_{ij}$.
Then the claim follows from the uniqueness part of proposition~\ref{parabolicinduction}.
\end{proof}

\begin{rk}\label{Kinertial}
A compatibility of this kind is implicit in~\cite[Remarque~5.17]{MStypes}.
\end{rk}

\subsection{Level zero maps for simple inertial classes.}
We continue to fix an endo-class~$\Theta_F$ and a lift~$\Theta_E \to \Theta_F$.
Let~$\kappa$ be a conjugacy class of $\beta$-extensions of maximal simple characters in~$G$ of endo-class~$\Theta_F$.
\begin{defn} 
Define the level zero map
\begin{displaymath}
\Lambda_{R}(-, \Theta_E, \kappa): \left ( \text{simple inertial classes of~$R[G]$-representations with endo-class~$\Theta_F$} \right ) \to \Gamma(\Theta_F) \backslash X_R(\Theta_F) 
\end{displaymath}
by sending the inertial class of supercuspidal support $[\GL_{m/r}(D), \pi_0^{\otimes r}]$ to the inflation of $\Lambda_{R}(\pi_0, \Theta_E, \kappa_0)$ through the norm $\be_{n/\delta(\Theta_F)}^\times \to \be_{n/r\delta(\Theta_F)}^\times$.
Here~$\kappa_{0}$ is the maximal $\beta$-extension in~$\GL_{m/r}(D)$ compatible with~$\kappa$.
\end{defn}

\begin{pp}
Let~$\pi$ be a cuspidal, non-supercuspidal irreducible representation of~$\GL_m(D)$. 
Then the two definitions we have given for the level zero part of~$\pi$ coincide.
\end{pp}
\begin{proof}
Choose an element of~$\Psi(\Theta_E)$ and then an irreducible factor~$\sigma$ of~$\bK_\kappa(\pi)$, viewed as a cuspidal representation of~$\GL_{m'}(\be_{d'})$. 
Then~$\sigma$ corresponds to an orbit~$[\chi]_{d'}$ of characters $\be_{n/\delta(\Theta_F)}^\times \to R^\times$ under~$\Gal(\be_{n/\delta(\Theta_F)} / \be_{d'})$. Our first definition yields $\Lambda_{R}(\pi, \Theta_E, \kappa) = [\chi]$, the orbit under~$\Gal(\be_{n/\delta(\Theta_F)} / \be)$. 

Since~$\pi$ is not supercuspidal, neither is~$\sigma$, and so $\chi$ is not itself $\be_{d'}$-regular. 
Hence it is norm-inflated from an $\be_{d'}$-regular character~$\chi_s$ of some intermediate $\be_{n/\delta(\Theta_F)s}^\times$. 
The supercuspidal support of~$\sigma$ corresponds to the $\Gal(\be_{n/\delta(\Theta_F)s} / \be_{d'})$-orbit of~$\chi_s$.

On the other hand, the supercuspidal support of~$\pi$ is inertially equivalent to some~$ \left ( \GL_{m/r}(D), \pi_0^{\otimes r} \right )$. 
The second definition of~$\Lambda_{R}(\pi, \Theta_E, \kappa)$ is the inflation of the character orbit $\Lambda_{R}(\pi_0, \Theta_E, \kappa_0)$. 
By construction, $\Lambda_{R}(\pi_0, \Theta_E, \kappa_0)$ is an orbit of~$\Gal(\be_{n/\delta(\Theta_F) r}/\be)$ on the $\be_{d'}$-regular characters of~$\be_{n/\delta(\Theta_F)r}^\times$.
By compatibility of~$\kappa$ and~$\kappa_{0}$, the supercuspidal support of~$\sigma$ consists of representations attached to $\Gal(\be_{n/\delta(\Theta_F) r}/\be_{d'})$-suborbits of $\Lambda_{R}(\pi_0, \Theta_E, \kappa_0)$. 
By uniqueness of supercuspidal support, it follows that $r = s$ and that the two definitions agree.
\end{proof}

Next we prove a result concerning reduction modulo~$\ell$.
It uses the fact that a $\beta$-extension~$\kappa$ of a maximal simple $\cbQ_\ell$-character~$\theta$ is integral and the reduction~$\br_\ell(\kappa)$ is a $\beta$-extension of~$\br_\ell(\theta)$, which is a maximal simple $\cbF_\ell$-character (this is proved in~\cite[Proposition~2.37]{MStypes}).

\begin{lemma}\label{modlreduction}
Let~$\pi$ be an integral $\cbQ_\ell$-representation of~$\GL_m(D)$ which is simple of endo-class~$\Theta_F$. Fix a lift~$\Theta_E \to \Theta_F$ and a conjugacy class~$\kappa$. 
If~$\tau$ is an irreducible factor of~$\br_\ell(\pi)$, then 
\[
\Lambda_{\cbF_\ell}(\tau, \Theta_E, \br_\ell(\kappa)) = \Lambda_{\cbQ_\ell}(\pi, \Theta_E, \kappa)^{(\ell)}.
\]
\end{lemma}
\begin{proof}
The representation~$\pi$ is a subquotient of a parabolic induction $\chi_1\pi^0 \times \cdots \times \chi_r\pi^0$ for an integral supercuspidal representation~$\pi^0$ of some~$\GL_{m/r}(F)$ and unramified characters~$\chi_i$ valued in~$\overline{\bZ}_\ell^\times$.

Recall that~$\Lambda_{\cbF_\ell}(\tau, \Theta_E, \br_\ell(\kappa)) $ is the inflation of the character orbit corresponding to the supercuspidal support of any irreducible factor of $\bK_{\br_{\ell}(\kappa)}(\tau)$, under any isomorphism $J_\theta/J^1_\theta \to \GL_{m'}(\be_{d'})$ in the conjugacy class~$\Psi(\Theta_E)$. 
So it suffices to prove that the supercuspidal support of any irreducible factor of $\bK_{\br_{\ell}(\kappa)}(\br_\ell(\pi))$ is a product of representations attached to $\Lambda_{\cbQ_\ell}(\pi, \Theta_E, \kappa)^{(\ell), \reg}$, which is the same as $\Lambda_{\cbQ_\ell}(\pi, \Theta_E, \kappa)^{\reg, (\ell), \reg}$.

By~\cite[Lemme~5.11]{MStypes}, the equality $\br_\ell[\bK_{\kappa}(\pi)] = [\bK_{\br_\ell(\kappa)}(\br_\ell(\pi))]$ holds. 
If~$\sigma$ is an irreducible factor of~$\bK_\kappa(\pi)$, the supercuspidal support~$\scusp_{\cbQ_\ell}(\sigma)$ is contained in $[\bK_{\kappa_0}(\pi^0)]^{\otimes r}$, where~$\kappa_0$ is compatible with~$\kappa$ (this is part of the definition of compatibility). 
So, $[\sigma]$ is contained in the parabolic induction of $[\bK_{\kappa_0}(\pi^0)]^{\otimes r}$.
Hence, the reduction~$\br_\ell[\sigma]$ is contained in the parabolic induction of $\br_\ell[\bK_{\kappa_0}(\pi^0)]^{\otimes r}$.
It follows that the supercuspidal support of every irreducible factor of $\br_\ell[\bK_{\kappa}(\pi)]$ is the supercuspidal support of a factor of $\br_\ell[\bK_{\kappa_0}(\pi^0)]^{\otimes r}$.

By definition, $[\bK_{\kappa_0}(\pi^0)]$ goes under any isomorphism in~$\Psi(\Theta_E)$ to the sum of the representations attached to~$\Lambda_{\cbQ_\ell}(\pi, \Theta_E, \kappa)^{\reg}$. The supercuspidal support of the reduction of any of these representations is a multiple of a representation attached to~$\Lambda_{\cbQ_\ell}(\pi, \Theta_E, \kappa)^{\reg, (\ell), \reg}$. Hence the supercuspidal support of every factor of~$\br_\ell [\bK_{\kappa_0}(\pi^0)]^{\otimes r}$ is a product of representations attached to~$\Lambda_{\cbQ_\ell}(\pi, \Theta_E, \kappa)^{\reg, (\ell), \reg}$. 
This concludes the proof of the lemma.
\end{proof}

Finally, we determine how the map~$\Lambda_R(-, \Theta_E, \kappa)$ changes when we change the lift~$\Theta_E \to \Theta_F$ or the $\beta$-extension~$\kappa$.
If~$\kappa_1, \kappa_2$ are $\beta$-extensions of a maximal simple character~$\theta$, by~(\ref{twistbeta}) there exists a character~$\chi$ of~$\be^\times$ such that~$\kappa_2 \cong \chi \otimes \kappa_1$.
By definition of the $\bK$-functors, we find that
\begin{equation}\label{changebetaextension}
\Lambda_R(-, \Theta_E, \kappa_2) = \chi\Lambda_R(-, \Theta_E, \kappa_1).
\end{equation}
Similarly, let~$\gamma \in \Gal(E/F)$ and lift it to an automorphism of~$\be_{d'}$. 
Given an isomorphism $J_\theta/J^1_\theta \to \GL_{m'}(\be_{d'})$ in the conjugacy class~$\Psi(\Theta_E)$, induced by an embedding $E \to F[\beta]$ in a parameter field, one sees that the isomorphism $J_\theta/J^1_\theta \to \GL_{m'}(\be_{d'})\xrightarrow[]{\gamma}  \GL_{m'}(\be_{d'})$ is in the conjugacy class~$\Psi(\gamma^*\Theta_E)$. 
It follows that
\begin{equation}\label{changelift}
\gamma^*\Lambda_R(-, \gamma^*\Theta_E, \kappa) = \Lambda_R(-, \Theta_E, \kappa).
\end{equation}

\paragraph{Conclusion.}
Let us fix for every endo-class~$\Theta_F$ a lift~$\Theta_E \to \Theta_F$ and a conjugacy class~$\kappa(\Theta_F)$ of $\beta$-extensions of maximal simple characters in~$G = \GL_m(D)$.
We will always form level zero maps using this choice of lift, and so we will sometimes omit~$\Theta_E$ from the notation.
The following theorem says that our two invariants give a complete parametrization of simple inertial classes of~$R[G]$.

\begin{thm}\label{parametrization}
The map $ \inv: \fs \mapsto (\cl(\fs), \Lambda_{R}(\fs, \kappa(\cl \fs)))$ is a bijection from the set of simple inertial classes of $R[\GL_m(D)]$-representations to the set of pairs $(\Theta_F, [\chi])$ consisting of an endo-class~$\Theta_F$ of degree dividing~$n = md$ and a character orbit $[\chi] \in \Gamma(\Theta_F) \backslash X_R(\Theta_F)$.
\end{thm}
\begin{proof}
The map~$\inv$ determines whether a simple inertial class~$\fs$ is supercuspidal, in terms of the regularity of $\Lambda_R(\fs, \kappa(\cl \fs))$. 
A supercuspidal inertial class~$\fs$ is determined by the conjugacy class of maximal simple types it contains, which can be recovered from $\inv(\fs)$. 
To see this, write~$\inv(\fs) = (\Theta_F, [\chi])$. 

Let~$\theta$ be a maximal simple character in~$\GL_{m}(D)$ with endo-class~$\Theta_F$. Let~$\kappa$ be the $\beta$-extension of~$\theta$ in the conjugacy class~$\kappa(\Theta_F)$. (To see that~$\kappa$ is well-defined, notice that if two $\beta$-extensions of the same maximal simple character are $G$-conjugate then they coincide. Indeed, if~$g \in G$ normalizes~$\theta$, then it normalizes~$\kappa$, because the normalizer~$\bJ(\theta)$ of~$\theta$ in~$G$ normalizes $J_\theta$, which is the unique maximal compact subgroup of~$\bJ(\theta)$. But~$\theta$ and~$\kappa$ have the same $G$-intertwining, by a defining property of $\beta$-extensions.)

Then, by construction, the maximal simple types of~$\fs$ have the form $(J_\theta, \kappa \otimes \sigma)$, where $\sigma$ is any supercuspidal representation of~$\GL_{m'}(\be_{d'})$ in the orbit corresponding to~$[\chi]$, inflated to~$J_\theta/J^1_\theta$ under any element of~$\Psi(\Theta_E)$.
(To see directly that these types are conjugate in~$G$, let $[\fA, \beta]$ be a simple stratum defining~$\theta$, and  fix an $F[\beta]$-linear isomorphism $B \to M_{m'}(D')$. Then the normalizer $\bJ(\theta)$ equals~$\pi_{D'}^\bZ \ltimes J_\theta$, and conjugation by~$\pi_{D'}$ acts as the Frobenius element of~$\Gal(\bd' / \be)$ on~$\bd'$.)
 
Hence our map is injective when restricted to supercuspidal inertial classes, and its image consists of those $(\Theta_F, [\chi])$ such that~$[\chi]$ is $\be_{d'}$-regular. The result for all simple inertial classes follows from this. 
Indeed, assume that $\fs$ has supercuspidal support~$\fs_0^{\otimes r}$, and~$\kappa(\fs)_0$ is the conjugacy class of $\beta$-extensions in~$\GL_{m/r}(D)$ compatible with~$\kappa(\fs)$. 
Then 
\[
\cl(\fs) = \cl(\fs_0) \text{ and } \Lambda_R(\fs, \kappa(\fs))^\reg = \Lambda_R(\fs_0, \kappa(\fs)_0)
\]
and the claim follows, since we see that~$\inv(\fs)$ determines the inertial supercuspidal support of~$\fs$.
\end{proof}

\begin{rk}
We emphasize again that the construction of the map~$\fs \mapsto \Lambda_R(\fs, \kappa(\fs))$ depends on the choice of a lift of~$\cl(\fs)$ to its unramified parameter field. 
Equivalently, we can interpret the inverse to~$\inv$ as a map 
\[
(\Theta_F, \Theta_E, [\chi]) \mapsto \fs_G(\Theta_F, \Theta_E, [\chi]), 
\]
with finite fibers consisting of orbits of $\Gal(\be / \mbf)$ acting diagonally by $g\cdot(\Theta_F, \Theta_E, [\chi]) = (\Theta_F, g^*\Theta_E, (g^{-1})^*[\chi])$. 

The triple~$(\Theta_F, \Theta_E, [\chi])$ corresponds to a supercuspidal inertial class of~$\GL_n(F)$ if and only if~$[\chi]$ consists of~$\be$-regular characters of~$\be_{n/\delta(\Theta_F)}^\times$. 
If this happens, then $\fs_G(\Theta_F, \Theta_E, [\chi])$ is supercuspidal for all inner forms~$G$ of~$\GL_n(F)$. 
When the inner form is~$D^\times$ for a division algebra~$D$, one has~$\be_{n/ \delta(\Theta_F)} = \be_{d'}$, so every triple is supercuspidal for~$D^\times$---of course, this is as expected because $D^\times$ has no nontrivial rational parabolic subgroups and so every irreducible smooth representation is supercuspidal.
\end{rk}

\begin{rk}\label{compareendo-classes} In~\cite{BSSV} there is assigned an endo-class~$\Theta(\fs)$ to every simple inertial class~$\fs$ of complex representations of~$\GL_m(D)$, defined to be the endo-class of any simple character contained in representations in~$\fs$. Since~$\fs$ need not be supercuspidal, these characters need not be maximal simple characters, but if $\fs_0$ is supercuspidal then $\Theta(\fs_0) = \cl(\fs_0)$ by definition. 
In addition, one knows that $\Theta(\fs) = \Theta(\fs_0)$ if~$\fs$ is inertially equivalent to a multiple of~$\fs_0$. 
(This fact is implicit in the construction of compatible $\beta$-extensions. See also the proof of~\cite[Theorem~6.6]{SecherreStevensJL}).) 
Then by construction we see that $\cl(\fs) = \Theta(\fs)$ for every simple inertial class~$\fs$.
\end{rk}

\begin{rk}\label{invariants}
In~\cite[Section~3.4]{MStypes} there is defined a number of invariants attached to a cuspidal representation~$\rho$ of~$\GL_m(D)$. If~$(J, \kappa \otimes \sigma)$ is a maximal simple type in~$\rho$, these are
\begin{enumerate}
\item $n(\rho)$, the \emph{torsion number}, which is the number of unramified characters~$\chi$ of~$G$ such that $\rho \otimes \chi \cong \chi$
\item $b(\rho)$, the size of the orbit of~$\sigma$ under the action of~$\Gal(\be_{d'}/\be)$
\item $s(\rho)$, the order of the stabilizer of~$\sigma$ in~$\Gal(\be_{d'}/\be)$
\item $f(\rho) = n/e(\Theta_F)$.
\end{enumerate}
These only depend on the inertial class of~$\rho$ and can be read off from our parametrization. We make this explicit over the complex numbers. Write $\inv(\rho) = (\Theta_F, [\chi])$. We have the equality $f(\rho) = n(\rho)s(\rho)$, by an explicit computation using~\cite[2.6.2(4)(b)]{BHJL}  (or see~\cite[Equation~(3.6)]{MStypes}). 
We also note that~\cite[Section~2]{BHJL} defines a \emph{parametric degree} for all simple representations, and for a supercuspidal~$\rho$ this equals~$n/s(\rho)$ (see~\cite[Section~3.1]{SecherreStevensJL}). 

The stabilizer~$S_1$ of any representative~$\chi$ of~$[\chi]$ under~$\Gal(\be_{n/\delta(\Theta_F)}/ \be)$ is isomorphic to the stabilizer~$S_2$ of the corresponding cuspidal representation of~$\GL_{m'}(\be_{d'})$ under~$\Gal(\be_{d'}/ \be)$. (Indeed, $S_1$ surjects onto~$S_2$ by restriction, and $S_1 \cap \Gal(\be_{n/\delta(\Theta_F)} / \be_{d'}) = 1$, because~$[\chi]$ consists of $\be_{d'}$-regular characters.) The quantity~$s(\rho)$ therefore also equals~$s[\chi]$, the order of the stabilizer of any element of~$[\chi]$ under~$\Gal(\be_{n/\delta(\Theta_F)}/ \be)$. We will also denote by~$b[\chi]$ the size of the orbit under~$\Gal(\be_{d'}/\be)$ of any representation of~$\GL_{m'}(\be_{d'})$ of the form~$\sigma(\chi)$ for $\chi \in [\chi]$.
\end{rk}

\section{The inertial Jacquet--Langlands correspondence.}
We now proceed to our main theorems. 
As in the previous section, we fix lifts $\Theta_E \to \Theta_F$ of all endo-classes of degree dividing~$n$. 
From now on, for any endo-class~$\Theta_F$ we will work with the conjugacy class~$\kappa$ of $p$-primary maximal $\beta$-extensions in any inner form of~$\GL_n(F)$, which then determines a conjugacy class in any~$\GL_m(D)$ uniquely by compatibility. 
This compatible conjugacy class, however, may not be $p$-primary.
We will work over the complex numbers unless stated otherwise, and so we will sometimes shorten notation to~$\fs \mapsto \Lambda_\kappa(\fs)$ for the level zero map $\fs \mapsto \Lambda_\bC(\fs, \Theta_E, \kappa)$.

\subsection{A character formula.}
Consider a supercuspidal irreducible representation~$\pi$ of~$G$. 
Let $\fs$ be the inertial class of~$\pi$ and let~$\kappa$ be the $p$-primary $\beta$-extension of endo-class~$\Theta_F = \cl(\fs)$ in~$G$. 
Let $(J_\theta, \lambda)$ be a maximal simple type for~$\fs$, so that~$\pi$ is the compact induction of an extension~$\Pi$ of~$\lambda$ to its normalizer~$\bJ(\lambda)$. 
Write~$[\chi]$ for the level zero part~$\Lambda_\kappa(\fs)$.
The type $(J_\theta, \lambda)$ is constructed from a maximal simple character~$\theta$ with endo-class~$\Theta_F$, and we fix a simple stratum $[\fA, \beta]$ defining~$\theta$. 
Let $T = \Fbetaur$, an unramified parameter field for~$\theta$. 

Write~$B = Z_A(F[\beta]) \cong M_{m'}(D')$ for the commutant of $F[\beta]$ in~$A$. Fix an extension $K^+/F[\beta]$ in~$B$ that has maximal degree, is unramified and normalizes the order~$\fA$. Such a~$^+$ exists by the arguments in the proof of proposition~\ref{conjugateembeddings}. Consider the maximal unramified extension~$K$ of~$F$ in~$K^+$, write~$A_K$ for the commutant~$Z_A(K)$, and let~$G_K = A_K^\times$. In this context, the normalizer~$N_G(K)$ acts on~$G_K$, and there is an isomorphism
\begin{displaymath}
N_G(K)/G_K \to \Gal(K/F)
\end{displaymath}
by the conjugation action on~$K$. It follows that $\Gal(K/F)$ has a right action on isomorphism classes of representations of~$G_K$: if~$\tau$ is a representation and $t_\alpha \in N_G(K)$ maps to $\alpha \in \Gal(K/F)$, denote by $\tau^\alpha$ the representation $g \mapsto \tau(t_\alpha g t_\alpha^{-1})$. The isomorphism class of~$\tau^\alpha$ is independent of the choice of preimage~$t_\alpha$ of~$\alpha$. If~$\tau$ has endo-class $\Theta_K$, then $\tau^\alpha$ has endo-class $\alpha^*\Theta_K$.

Since~$\beta$ commutes with~$K$ and generates a field~$K^+ = K[\beta]$ over~$K$, and $(K^+)^\times \subseteq \fK(\fA)$, proposition~\ref{interiorlifting} applies and~$\theta$ has an interior $K$-lift~$\theta_K$. This is a character of~$H^1_K = H^1_\theta \cap B$, and it is defined by the simple stratum $[\fA_K, \beta]$ for $\fA_K = \fA \cap A_K$. As in lemma~\ref{maximallift}, it is a maximal simple character.

Take the Heisenberg representation $\eta_K$ of~$J^1_K$ attached to~$\theta_K$. Let $\kappa_K$ be its $p$-primary $\beta$-extension to~$J_K$. Then, if we set~$\lambda_K = \kappa_K$ we obtain a maximal simple type in~$G_K$: it corresponds to the trivial character of $J_K/J^1_K \cong \mu_K = \mu_{K^+}$.

By the discussion after theorem~\ref{types}, the normalizer~$\bJ(\lambda_K)$ equals $(K^+)^\times J_K$. The representation~$\lambda_K$ extends to $K^\times J_K = \pi_K^\bZ \times J_K$ by letting~$\pi_K$ act trivially, and since $\bJ(\lambda_K) / K^\times J^1_K$ is cyclic of order~$e(F[\beta] / F)$ it also extends to~$\bJ(\lambda_K)$. However, as $F[\beta] / F$ might be wildly ramified, we cannot normalize the extension via the order of the determinant as in proposition~\ref{p-primarywideextension}. 
We will refer to any representation obtained by compactly inducing one of these extensions from $\bJ(\lambda_K)$ to~$G_K$ as a \emph{$K$-lift} of~$\pi$. 
These are all supercuspidal irreducible representations of~$G_K$ with the same inertial class, hence the ambiguity in the definition will not affect arguments concerning the inertial class.

\paragraph{}
Let~$\tau$ be a $K$-lift of~$\pi$. The representations~$\pi$ and~$\tau$ are related via their characters, through a formula due to Bushnell and Henniart in the context of essentially tame endo-classes (see~\cite[Section~6]{BHJL}). 
Recall that~$\pi$ contains a maximal simple type $(J_\theta, \lambda)$ constructed from~$[\fA, \beta]$, and that $T = \Fbetaur$.

\begin{thm}\label{characterformula} Let $\zeta \in \mu_K$ generate the field~$K$ over~$F$, and let~$u$ be an elliptic, regular and pro-unipotent element of~$G_K$. Then 
\begin{displaymath}
\tr \pi(\zeta u) = (-1)^{m' + 1}s[\chi]^{-1}\epsilon_\theta(\zeta) \sum_{\alpha \in \Gal(\bk/\mbf)} \chi(\zeta^\alpha)\tr \tau^\alpha(u)
\end{displaymath}
where~$\chi$ is evaluated at~$\zeta$ via any $\be$-linear isomorphism $\iota: \bk \to \be_{n/\delta(\Theta_F)}$, where~$\bk$ is an $\be$-algebra via $\iota(\Theta_E)_T : \be \to \bt$.
\end{thm}

\begin{rk}
It is not immediate that the formula makes sense as written, but we will see while proving the theorem that the characters of~$\tau$ and~$\tau^\alpha$ coincide on~$u$ whenever~$\alpha \in \Gal(\bk/\bt)$, hence the right hand side is independent of the choice of representatives of~$[\chi]$ and of the choice of~$\iota$. 
Recall that the sign~$\epsilon_\theta(\zeta)$ has been defined before proposition~\ref{symplecticcomputation}, and that~$m'$ is defined by the existence of an isomorphism $B = Z_A(F[\beta]) \cong M_{m'}(D')$. 
A \emph{pro-unipotent} element~$u$ of~$G$ is one for which $u^{p^n} \to 1$ as $n \to + \infty$. See remark~\ref{invariants} for the definition of~$s[\chi]$.
\end{rk}

\begin{proof}
By the following lemma, the Harish-Chandra character of~$\pi$ at~$\zeta u$ can be computed by the Mackey formula for an induced representation
\begin{displaymath}
\tr\, \pi(\zeta u) = \sum_{y \in \bJ(\lambda) \backslash G} \tr\, \Pi (y \zeta u y^{-1}),
\end{displaymath}
as in \cite[Section~1.2]{BHJL} and~\cite[Appendix]{BHliftingI}.

\begin{lemma}
The element $\zeta u \in G$ is elliptic and regular over~$F$.
\end{lemma}
\begin{proof}
Since~$F[\zeta u]$ is a finite-dimensional $F$-vector space, it is complete, and so it contains~$\zeta$, which is the limit of a sequence of elements of the form~$(u\zeta)^{p^n}$.
But then it contains~$u$, hence $F[\zeta u] = K[u]$ is a maximal field extension of~$F$ in~$A$.
\end{proof}

The theorem will follow from a careful study of this formula for~$\tr\, \pi(\zeta u)$.

\begin{lemma}\label{conjugacylemma}
If $y \in G$ and $y\zeta uy^{-1} \in \bJ(\lambda)$, then $y\zeta uy^{-1} \in J_\theta$ and there exists $\widetilde{y}$ in the normalizer~$N_G(K)$ such that $\bJ(\lambda) y = \bJ(\lambda) \widetilde{y}$. For any such $\widetilde{y}$, one has $\widetilde{y} u\widetilde{y}^{-1} \in J^1_K$.
\end{lemma}
\begin{proof}

Since the valuation of the determinant of~$\zeta u$ is zero, and $\bJ(\lambda) / J_\theta$ is infinite cyclic generated by some power of a uniformizer of~$D'$, necessarily $y \zeta u y^{-1} \in J_\theta$ if $y \zeta u y^{-1} \in \bJ(\lambda)$. The quotient $J_\theta / J^1_\theta$ is isomorphic to a general linear group $\GL_{m'}(\be_{d'})$, and the degree~$[K:F]$ equals~$n/e(\Theta_F)$, so~$\bk^\times$ embeds in~$J_\theta/J^1_\theta$ as a maximal elliptic torus. Now the claim follows as in the proof of~\cite[Lemma~13]{BHLLIII}: first prove that $y\zeta y^{-1} \in J_\theta$ by raising to a suitable power of~$p$, and then notice that there exists some other~$\zeta' \in \mu_K$ generating~$K$ over~$F$ with $y\zeta y^{-1}$ conjugate in~$J_\theta$ to~$\zeta'u'$ for some~$u' \in J_\theta^1$. 
By~\cite[ 2.6 Conjugacy Lemma]{BHeffective}, or \cite[Lemma~14]{BHLLIII}, we can further change~$y$ in its $J_\theta$-coset and assume that~$u'$ and~$\zeta'$ commute, and this implies that~$u' = 1$. But then $y\zeta u y^{-1} = \zeta' y u y^{-1}$ with $yuy^{-1}$ commuting with $\zeta'$ and contained in~$J_\theta$. As the image of~$\zeta'$ in~$J_\theta / J^1_\theta$ is a regular elliptic element, it commutes with no unipotent elements except the identity, so $yuy^{-1} \in J^1_\theta$.
\end{proof}

\begin{lemma}\label{normalizeraction}
The group~$\bJ(\lambda) \cap G_K$ equals~$\bJ(\lambda_K)$, and the order of the image of $\bJ(\lambda) \cap N_G(K)$ under the isomorphism $N_G(K)/G_K \to \Gal(K/F)$ equals $n/\delta(\Theta_F)b[\chi] = m'd'/b[\chi]$, where $b[\chi]$ equals the index of~$\bJ(\lambda)$ in $\pi_{D'}^\bZ \ltimes J_\theta$.
\end{lemma}
\begin{proof}
Compare~\cite[Proposition~9]{BHLLIII}.
We can determine an element in~$\Gal(K/F)$ by its action on~$\mu_K$. Any choice of isomorphism $\psi: J_\theta/J^1_\theta \to \GL_{m'}(\be_{d'})$ in~$\Psi(\Theta_E)$ induces a surjective group homomorphism
\begin{displaymath}
\widetilde{\psi}: \pi_{D'}^\bZ \ltimes J_\theta \to \Gal(\be_{d'} / \be) \ltimes \GL_{m'}(\be_{d'}) 
\end{displaymath}
which sends~$\pi_{D'}$ to a generator of~$\Gal(\be_{d'}/\be)$ and maps~$\mu_{K}$ isomorphically onto its image, which is an elliptic maximal torus in~$\GL_{m'}(\be_{d'})$, hence self-centralizing in $\Gal(\be_{d'}/\be) \ltimes \GL_{m'}(\be_{d'})$ (to see this, embed this group in~$\GL_{m'd'}(\be)$, where the image of~$\mu_K$ is still an elliptic maximal torus). 
So, if $x \in \pi_{D'}^\bZ \ltimes J_\theta$ centralizes~$\mu_K$ then it is contained in $\pi_{D'}^{d'\bZ} \ltimes J_\theta$, which equals $\pi_{F[\beta]}^\bZ \ltimes J_\theta$ since $\mO_{D'}^\times \subseteq J_\theta$. 
This implies that
\begin{displaymath}
\bJ(\lambda) \cap Z_G(K) = (\pi_{F[\beta]}^\bZ \ltimes J_\theta) \cap Z_G(K) = \pi_{F[\beta]}^\bZ \times J_K = \bJ(\lambda_K),
\end{displaymath}
thereby proving the first claim of the lemma.

The group~$\bJ(\lambda)$ has index~$b[\chi]$ in~$\pi_{D'}^\bZ \ltimes J_\theta$ and contains~$J_\theta$, hence maps under~$\widetilde{\psi}$ to $\Delta \ltimes \GL_{m'}(\be_{d'})$, for $\Delta \subseteq \Gal(\be_{d'}/\be)$ the only subgroup of index~$b[\chi]$.
The normalizer of~$\mu_K$ in~$\Delta \ltimes \GL_{m'}(\be_{d'})$ induces on~$\mu_K$ a group of automorphisms of order $m'd'/b[\chi]$: this is because $\mu_K$ arises up to $\GL_{m'}(\be_{d'})$-conjugation from restricting the scalars of $\be_{m'd'}$ to~$\be_{d'}$.
But any automorphism of~$\mu_K$ induced by a conjugation in $\Delta \ltimes \GL_{m'}(\be_{d'})$ is also induced by a conjugation in $\bJ(\lambda)$ and so the same holds for the normalizer of $\mu_K$ in $\bJ(\lambda)$, which is $\bJ(\lambda) \cap N_G(K)$.
(To see this, observe that if $x \in \pi_{D'}^\bZ \ltimes J_\theta$ and $x \zeta x^{-1} = \zeta' u$ for some $u \in J^1_\theta$ then we can change~$x$ in its $J_\theta$-coset and assume that~$\zeta'$ and~$u$ commute, by \cite[Lemma~14]{BHLLIII}. Then since the order of~$\zeta$ and~$\zeta'$ is prime to~$p$ and~$J_\theta^1$ is a pro-$p$ group we conclude that~$u = 1$, and the claim follows.)

\end{proof}

The space $\bJ(\lambda) N_G(K)$ decomposes into double cosets
\begin{displaymath}
\bJ(\lambda) N_G(K) = \bigcup_{\sigma \in \Gal(K/F)}\bJ(\lambda) t_\sigma G_K
\end{displaymath}
where $t_\sigma \in N_G(K)$ induces~$\sigma$ on~$K$ upon conjugation, and $\bJ(\lambda) t_\sigma G_K = \bJ(\lambda) t_\tau G_K$ if and only if $\tau \sigma^{-1}$ is induced by~$\bJ(\lambda)$. Then by lemma~\ref{conjugacylemma} and lemma~\ref{normalizeraction} we may rewrite the sum as
\begin{displaymath}
\tr \pi(\zeta u) = \sum_{y \in \bJ(\lambda) \backslash \bJ(\lambda) N_G(K)} \tr \Pi (y \zeta u y^{-1}) = (\delta(\Theta_F)b[\chi] / n)\sum_{\alpha \in \Gal(K/F)}\sum_{y \in \bJ(\lambda_K) \backslash G_K}\tr \Pi(y \alpha(\zeta u)y^{-1}).
\end{displaymath}
Here, $\alpha(\zeta u) = t_\alpha \zeta u t_\alpha^{-1}$. 

These~$y$ commute with all the $\alpha(\zeta )$. 
We are now going to fix an isomorphism $\psi: J_\theta/J^1_\theta \to \GL_{m'}(\be_{d'})$ in the conjugacy class~$\Psi(\Theta_E)$, and a representation~$\sigma$ of~$\GL_{m'}(\be_{d'})$ such that $\lambda = \kappa \otimes \psi^*\sigma$. 
Furthermore, we will choose a representative~$\chi$ for~$[\chi]$ such that $\sigma = \sigma(\chi)$ under the Green parametrization. 
(Recall that~$[\chi]$ only determines a $\Gal(\be_{d'}/\be)$-orbit of representations of~$\GL_{m'}(\be_{d'})$. Via the choice of~$\psi$ and~$\sigma$, we are fixing an element of this orbit.) 
Then
\begin{displaymath}
\tr \Pi (y\alpha(\zeta u)y^{-1}) = \tr \Pi(\alpha(\zeta) y\alpha(u)y^{-1}) = \tr \sigma(\zeta^\alpha) \tr \kappa(\zeta^\alpha y\alpha(u)y^{-1})
\end{displaymath}
since~$\Pi$ extends~$\lambda$, where~$\zeta^\alpha = \alpha(\zeta)$ and~$\sigma$ is evaluated at~$\zeta^\alpha$ via~$\psi$.

\begin{lemma}\label{Glauberman correspondence}
The equality
\begin{displaymath}
\tr \kappa(\zeta^\alpha y\alpha(u)y^{-1}) = \epsilon_\theta(\zeta^\alpha)\tr \lambda_K(y\alpha(u)y^{-1})
\end{displaymath}
holds whenever $y\alpha(u)y^{-1} \in J^1_K$.
\end{lemma}
\begin{proof}
Compare~\cite[Section~5.2]{BHeffective}. 
We use the Glauberman correspondence of proposition~\ref{pp:Glaubermancorrespondence} for the cyclic group $\Gamma \subseteq\mu_K$ generated by~$\zeta$, acting on~$J^1_\theta/\ker(\theta)$ and normalizing~$\eta$. 
This implies that there exist a unique irreducible representation~$\eta^\Gamma$ of $(J^1/\ker(\theta))^\Gamma$ and sign $\epsilon_\theta(\zeta) = \pm 1$ such that
\begin{equation}\label{eqn:Glaubermancorrespondence}
\tr \widetilde{\eta}(\zeta x) = \epsilon_\theta(\zeta)\tr \eta^\Gamma(x)
\end{equation} 
for all $x \in (J^1/\ker(\theta))^\Gamma$ and every generator~$\zeta$ of~$\Gamma$. 

Recall from section~\ref{maximalsimpletypesI} that~$\widetilde{\eta}$ is the only irreducible representation of $\Gamma \ltimes J^1_\theta$ with trivial determinant on~$\Gamma$. 
Hence, by construction, $\widetilde{\eta}$ is isomorphic to the restriction of the $p$-primary $\beta$-extension~$\kappa$ to $\Gamma \ltimes J^1_\theta$, since $\det \kappa$ has order a power of~$p$ and~$\Gamma$ has order prime to~$p$. 
On the other hand, by proposition~\ref{pp:Glaubermancorrespondence} we have~$\eta^\Gamma = \eta_K$, the Heisenberg representation associated to~$\theta_K$.
To complete the proof of the lemma it suffices to notice that since~$\zeta$ generates~$K$ over~$F$, $\ker(\theta)$ is a pro-$p$ group and $\Gamma$ has order prime to~$p$, the argument in~\cite[Proposition~6]{BHLLIII} implies that $(J^1_\theta/\ker(\theta))^\Gamma = J^1_K/\ker(\theta_K)$.

\end{proof}

If $y\alpha(u)y^{-1} \not \in J^1_K$, then~$y\alpha(\zeta u) y^{-1} \not \in \bJ(\lambda)$ by lemma~\ref{conjugacylemma}. So we have
\begin{displaymath}
\tr \Pi (y\alpha(\zeta u)y^{-1}) = \epsilon_\theta(\zeta^\alpha)\tr \sigma(\zeta^\alpha)\tr \lambda_K(y\alpha(u)y^{-1})
\end{displaymath}
where the traces of~$\Pi$ and~$\lambda_K$ are extended by zero to~$G$ and~$G_K$ respectively. Since $\tau$ is induced from an extension of~$\lambda_K$ to~$\bJ(\lambda_K)$, we deduce that
\begin{displaymath}
\sum_{y \in \bJ(\lambda_K) \backslash G_K} \tr \lambda_K(y \alpha(u) y^{-1}) = \tr \tau (\alpha(u))
\end{displaymath}
and so
\begin{displaymath}
\tr \pi(\zeta u) =(\delta(\Theta_F)b[\chi] / n)\sum_{\alpha \in \Gal(\bk / \mbf)}\epsilon_\theta(\zeta^\alpha)\tr\sigma(\zeta^\alpha) \tr \tau(\alpha(u)).
\end{displaymath}

Now the Galois twists~$\tau^\alpha = \tau \circ \ad(t_\alpha)$ have character $x \mapsto \tr \tau(\alpha(x))$, and the endo-class~$\cl(\tau^\alpha)$ of~$\tau^\alpha$ satisfies $\cl(\tau^\alpha) = \alpha^*(\cl \tau)$. 
By~\cite[1.5.1]{BHliftingIV}, the group $\Gal(\bk / \mbf)$ is transitive over the $K$-lifts of~$\Theta_F$, and there is the same number of these as simple components of~$K \otimes_F F[\beta]$. 
So the stabilizer of $\cl(\tau)$ in~$\Gal(\bk / \mbf)$ is $\Gal(\bk / \bt)$, and it follows that for $\alpha \in \Gal(\bk / \bt)$ the supercuspidal representations~$\tau$ and~$\tau^\alpha$ of~$G_K$ have the same endo-class. 
By proposition~\ref{conjugateembeddings}, they both contain the simple character~$\theta_K$, so their restriction to~$J^1_K$ contains~$\eta_K$. 
Since $J_K / J_K^ 1 \cong \mu_{K^+} = \mu_K$, which by construction acts trivially on~$\tau$ and~$\tau^\alpha$, these representations contain the same maximal simple type $(J_K, \kappa_K)$. 
So they are inertially equivalent, and their characters therefore agree on elements whose reduced norm has valuation~0, such as~$u$. 
We can now rearrange the sum further to
\begin{displaymath}
\tr \pi(\zeta u) = (\delta(\Theta_F)b[\chi] / n)\epsilon_\theta(\zeta)\sum_{\gamma \in \Gal(\bt / \mbf)}\left ( \tr \tau^\gamma(u)\sum_{\delta \in \gamma\Gal(\bk / \bt)}\tr\sigma(\zeta^\delta) \right )
\end{displaymath}
since $\epsilon_\theta(\zeta)$ only depends on the subgroup of~$\mu_K$ generated by~$\zeta$.

The trace $\tr \sigma(\zeta^\delta)$ can be computed as follows. 
We are evaluating~$\sigma$ at~$\zeta^\delta$ using a fixed choice of isomorphism $\psi: J_\theta/J^1_\theta \to \GL_{m'}(\be_{d'})$ in~$\Psi(\Theta_E)$. 
Recall that any such~$\psi$ comes from an isomorphism $\psi: \fj_\theta / \fj^1_\theta \to M_{m'}(\be_{d'})$ by passing to groups of units. The elliptic maximal torus $\psi(\mu_K)$ is conjugate to~$\be_{m'd'}^\times$, where the trace of~$\sigma = \sigma(\chi)$ is given by explicit formulas, and this  isomorphism $\mu_K \to \be_{m'd'}^\times$ ($\psi$ followed by conjugation) comes from an $\be$-linear isomorphism $\bk \to \be_{m'd'}$ by passing to groups of units. Then one has the character formula
\begin{align*}
\sum_{\delta \in \Gal(\bk / \bt)}\tr \sigma(\zeta^{\delta}) & = (-1)^{m'+1}\sum_{\delta \in \Gal(\bk / \bt)}\sum_{\nu \in \Gal(\be_{m'd'} / \be_{d'})}\chi(\zeta^{\delta\nu})\\
& = (-1)^{m' +1} \sum_{\nu \in \Gal(\be_{m'd'} / \be_{d'})}\sum_{\chi_0 \in [\chi]}\chi_0(\zeta^{ \nu})\\
& = (-1)^{m'+1}m' \sum_{\delta \in \Gal(\bk / \bt)} \chi(\zeta^{ \delta})
\end{align*}
where~$\sigma$ is evaluated on~$\zeta^\delta$ via~$\psi$ and~$\chi$ is evaluated on~$\zeta^\delta$ via any $\be$-linear isomorphism $\iota: \bk \to \be_{m'd'}$. Because the sums are taken over $\Gal(\bk / \bt)$, the answer is independent of the choice of~$\psi$ and~$\iota$, and the second line shows that the answer does not depend on the choice of~$\chi$ in~$[\chi]$.

Now since $n/\delta(\Theta_F) = m'd'$ we have
\begin{displaymath}
(\delta(\Theta_F)b[\chi]/n)m' = b[\chi]/d' = s[\chi]^{-1},
\end{displaymath} 
and rearranging further we obtain
\nopagebreak
\begin{align*}
\tr \pi(\zeta u) & =  (-1)^{m'+1}s[\chi]^{-1}\epsilon_\theta(\zeta)\sum_{\gamma \in \Gal(\bt / \mbf)}\left ( \tr \tau^\gamma(u)\sum_{\delta \in \gamma\Gal(\bk/ \bt)}\chi(\zeta^{\delta}) \right )\\
& =  (-1)^{m'+1}s[\chi]^{-1}\epsilon_\theta(\zeta)\sum_{\gamma \in \Gal(\bk / \mbf)} \tr \tau^\gamma(u) \chi(\zeta^\gamma).
\end{align*}

\end{proof}

\subsection{Results from $\ell$-modular representation theory.}\label{lmodular}
Let $\ell \not = p$ be a prime number, and fix an isomorphism~$\iota : \bC \to \overline{\bQ}_\ell$. 
In~\cite[Section~4.1]{SecherreStevensJL} there is defined a notion of \emph{mod~$\ell$ inertial supercuspidal support} for irreducible smooth $\cbQ_\ell$-representations of~$G = \GL_m(D)$. It is an inertial class of supercuspidal supports for~$\GL_m(D)$ over~$\cbF_\ell$, and it only depends on the inertial class of the representation. Write~$\bi_\ell(\fs)$ for the mod~$\ell$ inertial supercuspidal support of the inertial class~$\fs$, and say that two classes $\fs_1$ and~$\fs_2$ for the category of $\overline{\bQ}_\ell$-representations of~$G$ are in the same \emph{$\ell$-block} if $\bi_\ell(\fs_1) = \bi_\ell(\fs_2)$.

Simple inertial classes~$\fs_i$ of complex representations of~$G$ are said to be \emph{$\ell$-linked} if the inertial classes over $\overline{\bQ}_\ell$ corresponding to them under~$\iota$ are in the same $\ell$-block. 
By \cite[Lemma~5.2]{SecherreStevensJL} this is independent of the choice of~$\iota$.
We say that the~$\fs_i$ are \emph{linked} if there exist prime numbers~$\ell_1, \ldots, \ell_r$ all distinct from~$p$ and inertial classes~$\fs^0,\ldots, \fs^r$ such that~$\fs^0 = \fs_0$, $\fs^r = \fs_1$, and~$\fs^{i-1}$ and~$\fs^i$ are~$\ell_i$-linked throughout.

By~\cite[Lemma~4.3]{SecherreStevensJL}, the mod~$\ell$ inertial supercuspidal support of an integral representation~$\pi$ coincides with the supercuspidal support of every factor of~$\br_\ell(\pi)$. From this and proposition~\ref{modlreduction} it follows that 
\begin{displaymath}
\cl(\bi_\ell(\fs)) = \cl(\fs)\; \text{and} \; \Lambda_{\cbF_\ell}(\bi_\ell(\fs), \br_\ell(\kappa)) = \Lambda_{\cbQ_\ell}(\fs, \kappa)^{(\ell)}.
\end{displaymath}
So, two simple inertial classes~$\fs_i$ are $\ell$-linked if and only if $\cl(\fs_1) = \cl(\fs_2)$ and $\Lambda(\fs_1, \kappa)^{(\ell)} = \Lambda(\fs_2, \kappa)^{(\ell)}$. Letting~$\ell$ vary, we see that the~$\fs_i$ are linked if and only if they have the same endo-class (compare~\cite[Propositions~5.5 and~5.8]{SecherreStevensJL}). The compatibility of the Jacquet--Langlands correspondence with respect to mod~$\ell$ reduction proved in~\cite{MScongruences} then implies the following result.

\begin{thm}\label{JLlinkedreps}\cite[Theorem~6.3]{SecherreStevensJL}
Let~$H = \GL_n(F)$ and consider the Jacquet--Langlands transfer of simple inertial classes of complex representations
\begin{displaymath}
\JL_G: \fB_{\mathrm{ds}}(G) \to  \fB_{\mathrm{ds}}(H)
\end{displaymath}
Let~$\ell \ne p$ be a prime number.
Let~$\fs_i$ be simple inertial classes for~$G$. Then~$\fs_1$ and~$\fs_2$ are $\ell$-linked if and only if~$\JL_G(\fs_1)$ and~$\JL_G(\fs_2)$ are $\ell$-linked.
\end{thm}

\subsection{Proof of the main theorems.}
Consider central simple algebras $A_1 = M_n(F)$ and $A_2 = M_m(D)$ over~$F$, such that~$md = n$.
Write $\JL_{G_2}: \mathbf{D}(G_2) \to \mathbf{D}(G_1)$ for the Jacquet--Langlands correspondence between their unit groups $G_i = A_i^\times$, as well as for the map it induces on simple inertial classes. Let~$\fs_i$ be simple inertial classes of the~$G_i$ with
\begin{displaymath}
\fs_1 = \JL_{G_2}(\fs_2).
\end{displaymath}

\begin{thm}\label{endo-classinvariance}
We have the equality $\cl(\fs_1) = \cl(\fs_2)$.
\end{thm}
\begin{proof}
Let~$d$ be the reduced degree of~$D$ over~$F$. For all integers $a \geq 1$ there exist simple inertial classes~$\fs_i^*$ in~$\GL_{an}(F)$ and~$\GL_{am}(D)$ which correspond to each other under the Jacquet--Langlands transfer on these groups and have endo-class~$\cl(\fs_i^*) = \cl(\fs_i)$: it suffices to let their supercuspidal support be a multiple of the supercuspidal support of the~$\fs_i$. Letting~$a = d$, we can assume that~$d$ divides $\frac{n}{\delta(\cl \fs_i)}$ for all~$i = 1,2$. 

We can assume that both~$\fs_i$ are supercuspidal: to see this, recall that the parametric degree of a simple inertial class is preserved under the Jacquet--Langlands correspondence~\cite[2.8 Corollary~1]{BHJL}. 
Since a simple inertial class of~$\GL_n(F)$ is supercuspidal if and only if it has maximal parametric degree, we see that the transfer of a supercuspidal representation of~$\GL_n(F)$ is supercuspidal. Let~$\fs_1^+$ be a supercuspidal inertial class for~$\GL_n(F)$, with $\cl(\fs_1) = \cl(\fs_1^+)$. Then $\fs_1$ and~$\fs_1^+$ are linked, hence by theorem~\ref{JLlinkedreps} also~$\fs_2$ and~$\JL_{G_2}^{-1}(\fs_1^+)$ are linked, and so they have the same endo-class. So if $\cl(\JL_{G_2}^{-1}(\fs_1^+)) = \cl(\fs_1^+)$ then the theorem follows. 

When the~$\fs_i$ are supercuspidal, their parametric degree is maximal and by the formulas in remark~\ref{invariants} the invariants~$s[\chi_i]$ are equal to~$1$, where $[\chi_i] = \Lambda_{\kappa}(\fs_i)$. Fix maximal simple characters~$\theta_i$ in~$G_i$ of endo-class~$\cl(\fs_i)$, with underlying simple strata $[\fA_i, \beta_i]$. Let $T_i = F[\beta_i]^{\mathrm{ur}}$. Let~$K_i^+$ be an extension of~$F[\beta_i]$ contained in~$Z_{A_i}(F[\beta_i])$  which has maximal degree, is unramified, and normalizes~$\fA_i$, as in the proof of proposition~\ref{conjugateembeddings}. Let~$K_i$ be the maximal unramified extension of~$F$ contained in~$K_i^+$. The quantity~$t = \frac{n}{\delta(\cl \fs_i)}f(F[\beta_i]/F)$ is preserved under the Jacquet--Langlands correspondence (it is the torsion number of the inertial class, by the formulas in remark~\ref{invariants}, which is preserved since the Jacquet--Langlands correspondence commutes with twists by unramified characters), and the~$K_i$ have the same degree~$t$ over~$F$.

Because~$d$ divides~$\frac{n}{\delta(\cl \fs_i)}$, it divides $[K_i:F]$, hence the commutant~$Z_{A_2}(K_2)$ is a split central simple algebra over~$K_2$: indeed, we have~$Z_{A_2}(K_2) \cong M_{m'}(D')$ for a central division algebra~$D'$ of reduced degree $d/(d, [K_2:F]) = 1$ over~$K_2$. We can therefore fix an $F$-linear isomorphism $\alpha_0: K_2 \to K_1$, and then an $\alpha_0$-linear isomorphism $\alpha: Z_{A_2}(K_2) \to Z_{A_1}(K_1)$. We emphasize that these isomorphisms are chosen arbitrarily; this ambiguity will not affect our results.  Using~$\alpha$ and~$\alpha_0$ to identify~$K_1$ and~$K_2$, and~$Z_{A_1}(K_1)$ and~$Z_{A_2}(K_2)$, we can write~$K$ for any of the~$K_i$ and~$A_K$ for any of the~$Z_{A_i}(K_i)$\footnote{Formally, $K$ is an inverse limit of the diagram $\alpha_0: K_2 \to K_1$, and similarly for~$A_K$ and~$\alpha: Z_{A_2}(K_2) \to Z_{A_1}(K_1)$.}.

Choose supercuspidal irreducible representations~$\pi_i$ in the inertial classes~$\fs_i$, corresponding to each other under the Jacquet--Langlands correspondence. Let~$\tau_i$ be some $K$-lift of~$\pi_i$. Choose $\zeta \in \mu_K$ generating~$K$ over~$F$, and an elliptic, regular and pro-unipotent element~$u$ of~$G_K = A_K^\times$. The $\zeta u$ are matching elements of~$A_1$ and~$A_2$, and by proposition~\ref{characterformula} and its proof, we have equalities
\begin{displaymath}
\tr \pi_i(\zeta u) = (-1)^{m_i'+1}\epsilon_{\theta_i}(\zeta)\sum_{\gamma \in \Gal(\bt_i / \mbf)}\left ( \tr \tau_i^\gamma(u)\sum_{\delta \in \gamma\Gal(\bk / \bt_i)}\chi_i (\zeta^{\delta}) \right ).
\end{displaymath}

By~\cite[1.5.1]{BHliftingIV}, the group $\Gal(\bk / \mbf)$ is transitive on the set of~$K$-lifts of~$\cl(\fs_i)$, which has $f(\cl\, \fs_i)$ many elements. Since $\cl(\tau_i^\gamma) = \gamma^*\cl(\tau_i)$, the representations~$\tau_i^\gamma$ as~$\gamma$ runs through $\Gal(\bt_i/\mbf)$ are pairwise inertially inequivalent (as they have different endo-classes). They are furthermore \emph{totally ramified} representations of~$G_K$, in the sense that their unramified parameter fields all coincide with~$K$.

\begin{lemma}[Linear independence lemma]
Let~$\pi_1, \ldots, \pi_r$ be irreducible, supercuspidal, totally ramified representations of~$\GL_m(D)$ for a central division algebra~$D$ over~$F$, whose central characters agree on~$\mu_F$. Assume that they are pairwise inertially inequivalent. Then the characters~$\tr \pi_i$ are linearly independent on the set of elliptic, regular, pro-unipotent elements of~$\GL_m(D)$.
\end{lemma}
\begin{proof}
This follows from~\cite[Lemma~6.6]{BHJL}, as we can twist the~$\pi_i$ by unramified characters of~$\GL_m(D)$ until the central characters also agree on a uniformizer of~$F$. This does not change the inertial classes of the~$\pi_i$, nor the character values on elliptic, regular, pro-unipotent elements of~$\GL_m(D)$ as these have reduced norms of valuation~0.
\end{proof}

The central characters of the~$\tau_i$ are trivial on~$\mu_K$ by construction. 
Then, by the linear independence lemma, either there exists~$\gamma \in \Gal(K/F)$ such that~$\tau_1^\gamma$ and~$\tau_2$ are inertially equivalent,  or 

\begin{displaymath}
\sum_{\delta \in \gamma\Gal(\bk / \bt_i)}\chi_i(\zeta^{\delta}) = 0
\end{displaymath}
for all values of~$i$, $\gamma$ and~$\zeta$. That this does not happen follows when $i = 1$ by~\cite[Theorem~1.1]{SZcharacters}, stating that there exists no character~$\chi$ of~$\bk^\times$ such that $\sum_{\gamma \in \Gal(\bk / \mbf)}\chi(\zeta^\gamma) = 0$ for all $\mbf$-regular elements of~$\bk$.

So we have proved that~$\tau_1$ and~$\tau_2^\gamma$ are inertially equivalent for some~$\gamma \in \Gal(K/F)$. But then they have the same endo-class, and since their endo-classes are~$K$-lifts of~$\cl(\fs_1)$ and~$\cl(\fs_2)$ respectively, the theorem follows.
\end{proof}

We pass now to the study of the level zero part of the~$\fs_i$. Let us first assume that the~$\fs_i$ are supercuspidal. Choose maximal simple characters~$\theta_i$ contained in the~$\fs_i$, defined by strata $[\fA_i, \beta_i]$, and let $T_i = F[\beta_i]^{\mathrm{ur}}$. As in the proof of proposition~\ref{conjugateembeddings}, we take a maximal unramified extension~$K_i^+$ of~$F[\beta_i]$ in~$Z_{A_i}(F[\beta_i])$ normalizing~$\fA_i$, and we identify the maximal unramified extensions~$K_i$ of~$F$ in~$K_i^+$. 
More precisely, we know that the~$\theta_i$ have the same endo-class, and we take the only $F$-linear ring isomorphism $\alpha_0: K_2 \to K_1$ such that
\begin{displaymath}
\alpha_0^*\cl(\theta_{1, K_1}) = \cl(\theta_{2, K_2})
\end{displaymath}
for the interior lifts~$\theta_{i, K_i}$. Notice, however, that the commutants $Z_{A_i}(K_i)$ need not be isomorphic. As in the proof of theorem~\ref{endo-classinvariance}, we write~$K$ for any of~$K_2$ and~$K_1$ and~$T$ for any of~$T_2$ and~$T_1$, using the isomorphism~$\alpha_0$.

\begin{thm}\label{comparecharacters}
The equality $\epsilon^1_{\theta_1}\Lambda_\kappa(\fs_1) = \epsilon^1_{\theta_2}\Lambda_\kappa(\fs_2)$ holds, where~$\epsilon^1_{\theta_i} =  \epsilon^1_{\mu_K}(-, V_{\theta_i})$ are the symplectic sign characters, and~$\kappa$ is the conjugacy class of $p$-primary $\beta$-extensions.
\end{thm}
\begin{proof}
We choose supercuspidal representations~$\pi_i$ of~$G_i$ in~$\fs_i$, Jacquet--Langlands transfers of each other, and we let~$\tau_i$ be some $K$-lift of~$\pi_i$; then, because of our choice of~$\alpha_0$, $\tau_1$ and~$\tau_2$ have the same endo-class.

Fix a root of unity~$\zeta \in \mu_K$ generating~$K$ over~$F$, and let~$u_1$ be an elliptic, regular, pro-unipotent element of~$G_{K, 1} = Z_{G_1}(K)$. The matching conjugacy class in~$G_{K, 2}$ then consists of pro-unipotent elements, as in the case of an elliptic regular element this is a condition which can be checked on the eigenvalues of the characteristic polynomial. Let~$u_2$ be an element of this conjugacy class. We apply proposition~\ref{characterformula} and obtain an equality 
\begin{displaymath}
\tr \pi_i(\zeta u_i) = (-1)^{m'_i+1}\epsilon_{\theta_i}(\zeta)\sum_{\gamma \in \Gal(\bt / \mbf)}\left ( \tr \tau_i^\gamma(u_i)\sum_{\delta \in \gamma\Gal(\bk / \bt)}\chi_i(\zeta^{\delta}) \right )
\end{displaymath}
where~$\Lambda_{\kappa}(\fs_i) = [\chi_i]$. By the linear independence lemma, we have that~$\tau_1^\gamma$ and~$\JL_{G_{K, 2}}(\tau_2)$ are inertially equivalent for some~$\gamma \in \Gal(\bt/ \mbf)$ (this is the Jacquet--Langlands correspondence for the groups~$G_{K, i}$, which are inner forms of each other). This~$\gamma$ is unique, as the~$\tau_i^\gamma$ have pairwise different endo-classes for~$\gamma \in \Gal(\bt / \mbf)$. By theorem~\ref{endo-classinvariance}, the endo-class of~$\JL_{G_{K, 2}}(\tau_2)$ is~$\cl(\theta_{2, K_2})$. By our choice of~$\alpha_0$, this implies~$\gamma = 1$. 

Fix~$u_i$ so that the characters of the~$\tau_i$ are nonzero at $u_i$; this is possible by the linear independence lemma, because the~$\tau_i$ are totally ramified. Then the Jacquet--Langlands character relation
\begin{displaymath}
(-1)^{n_K}\tr \tau_1(u_1) = (-1)^{m_K}\tr \tau_2(u_2)
\end{displaymath}
holds, where the signs are defined by the existence of isomorphisms $Z_{A_1}(K_1) \cong M_{n_K}(K)$ and $Z_{A_2}(K_2) \cong M_{m_K}(D_K)$ for some central division algebra~$D_K$ over~$K$.

We now deduce an equality
\begin{equation}\label{comparecharacters1}
(-1)^{m+m'+m_K+1}\epsilon_{\theta_2}(\zeta)\sum_{\delta \in \Gal(\bk / \bt)}\chi_2(\zeta^\delta) = (-1)^{n+n'+n_K+1}\epsilon_{\theta_1}(\zeta)\sum_{\delta \in \Gal(\bk / \bt)}\chi_1(\zeta^\delta)
\end{equation}
on comparing $\tr \pi_1(\zeta u_1)$ and $\tr \pi_2(\zeta u_2)$ by the Jacquet--Langlands correspondences over~$F$ and over~$K$. This equality holds for all~$\zeta \in \mu_K$ generating~$K$ over~$F$, or equivalently for all~$\zeta \in \bk^\times$ generating~$\bk$ over~$\mbf$. (To be more precise, recall that~$K$ is an inverse limit of the diagram $\alpha_0: K_2 \to K_1$. 
We are evaluating~$\chi_i$ at $\zeta^\delta \in \mu_{K_i}$ via a choice of $\be$-linear isomorphism $\iota_i: \bk_i \to \be_{n/\delta(\Theta_F)}$, as in theorem~\ref{characterformula}. 
Since $\alpha_0^*\cl(\theta_{1, K_1}) = \theta_{2, K_2}$, we have $\alpha_0\iota_{T_2} = \iota_{T_1}$, hence the~$\iota_i$ can be chosen compatibly with $\alpha_0 : \bk_2 \to \bk_1$, allowing us to evaluate $\chi_i$ to $\zeta^\delta \in \mu_K$.)

Now recall from our discussion of symplectic invariants that (writing~$V_i = V_{\theta_i}$) there exist a sign~$\epsilon^0_{\mu_K}(V_i)$ and a quadratic character~$\epsilon^1_{\mu_K}(z, V_i)$ of~$\mu_K$ such that, whenever $z \in \mu_K$ generates a subgroup~$\Delta$ of~$\mu_K$ with~$V_i^{\mu_K} = V_i^\Delta$, one has
\begin{displaymath}
\epsilon_{\theta_i}(z) = \epsilon^0_{\mu_K}(V_i)\epsilon^1_{\mu_K}(z, V_i).
\end{displaymath}
In our case, every~$\zeta$ generating~$K$ over~$F$ satisfies~$V_i^\zeta = V_i^{\mu_K}$ even if~$\zeta$ does not generate~$\mu_K$, by the argument of~\cite[Proposition~6]{BHLLIII}. 
Comparing coefficients, one gets an equality
\begin{displaymath}
(-1)^{n+n'+n_K+1}\epsilon^0_{\mu_K}(V_1)\sum_{\delta \in \Gal(\bk / \bt)}\epsilon^1_{\mu_K}(\zeta^\delta, V_1)\chi_1(\zeta^\delta) = (-1)^{m+m'+m_K+1}\epsilon^0_{\mu_K}(V_2)\sum_{\delta \in \Gal(\bk / \bt)}\epsilon^1_{\mu_K}(\zeta^\delta, V_2)\chi_2(\zeta^\delta)
\end{displaymath}
which we rewrite as
\begin{equation}\label{comparecharacters2}
(-1)^{n+n_K+m+m'+m_K+1}\epsilon^0_{\mu_K}(V_1)\epsilon^0_{\mu_K}(V_2)\sum_{\delta \in \Gal(\bk / \bt)}\epsilon^1_{\mu_K}(\zeta^\delta, V_1)\chi_1(\zeta^\delta) = (-1)^{n'+1}\sum_{\delta \in \Gal(\bk / \bt)}\epsilon^1_{\mu_K}(\zeta^\delta, V_2)\chi_2(\zeta^\delta).
\end{equation}
This equation stays true if $\zeta$ varies over all generators of the extension~$\bk / \mbf$.
Since $\fs_1$ is a supercuspidal inertial class for~$\GL_n(F)$, the character~$\chi_1$ is $\be$-regular. 
For any $\be$-regular character of~$\be_{n/\delta(\Theta_F)}^\times$ there exists a $\be$-regular character~$\chi_2$ of the same group so that
\[
\fs_{G_1}(\Theta_F, \Theta_E, [\chi_1]) = \JL(\fs_{G_2}(\Theta_F, \Theta_E, [\chi_2])).
\] 
This follows from theorem~\ref{endo-classinvariance} and the invariance of parametric degrees under the Jacquet--Langlands correspondence.
(Recall that~$\fs_{G_i}(\Theta_F, \Theta_E, [\chi_i])$ is the simple inertial class of~$G_i$ with invariants $(\Theta_F, [\chi_i])$.)
Then equation~\ref{comparecharacters2} continues to stay true for all such pairs~$(\chi_1, \chi_2)$. 

At the right hand side of~\ref{comparecharacters2}, one has the trace at~$\zeta$ of the supercuspidal irreducible representation $\sigma[\epsilon^1_{\mu_K}(-, V_2)\chi_2]$ of~$\GL_{n/\delta(\Theta_F)}(\bt)$. 
By~\cite[Corollary~1]{BHLLIII} we deduce that
\footnote{We could not apply this directly to equation~\ref{comparecharacters1} to deduce
\begin{align*}
 [\chi_1] & = [\chi_2]\\
(-1)^{n+n'+n_K+1}\epsilon_{\theta_1}(\zeta) & = (-1)^{m+m'+m_K+1}\epsilon_{\theta_2}(\zeta)
\end{align*}
because the sign~$\epsilon_{\theta_i}(\zeta)$ may not be constant on the~$\zeta$ that generate~$K$ over~$F$, since these may generate proper subgroups of~$\mu_K$.}
\begin{gather*}
(-1)^{n+n_K+m+m'+m_K+1}\epsilon^0_{\mu_K}(V_1)\epsilon^0_{\mu_K}(V_2) = (-1)^{n'+1} \text{ and }\\
\sigma [\epsilon^1_{\mu_K}(-, V_1)\chi_1]  \cong \sigma[\epsilon^1_{\mu_K}(-, V_2)\chi_2],
\end{gather*}
and the theorem follows.
\end{proof}

It follows from theorem~\ref{comparecharacters} that, twisting the $p$-primary $\beta$-extension by the symplectic sign character (a quadratic character), we obtain conjugacy classes~$\kappa_i$ of $\beta$-extensions in~$G_i$, of endo-class~$\Theta_F$, such that $\Lambda_{\kappa_1}(\fs_1) = \Lambda_{\kappa_2}(\fs_2)$ whenever the~$\fs_i$ are supercuspidal inertial classes and Jacquet--Langlands transfers of each other. Notice also that by~\cite[6.9]{BHJL}, the sign~$\epsilon^0$ and the character~$\epsilon^1$ determine each other: $\epsilon^1$ is the nontrivial quadratic character if and only if~$p$ is odd and~$\epsilon^0 = -1$. It follows that the quadratic character $\epsilon^1_{\mu_K}(V_1)\epsilon^1_{\mu_K}(V_2)$ is nontrivial if and only if~$p$ is odd and~$n+n'+n_K+m+m'+m_K$ is odd.

\begin{thm}\label{comparecharactersingeneral}
With the notation of the previous paragraph, the equality $\Lambda_{\kappa_1}(\fs_1) = \Lambda_{\kappa_2}(\fs_2)$ also holds for non-cuspidal $\fs_i$.
\end{thm}
\begin{proof}
This follows from theorem~\ref{comparecharacters} using the technique of~\cite[Lemma~9.11]{SecherreStevensJL}. Write~$[\chi_i]$ for~$\Lambda_{\kappa_i}(\fs_i)$ and assume that~$\chi_1$ is not $\be$-regular, as the $\be$-regular case has already been treated. Write~$\xi(\kappa_1, \kappa_2)$ for the permutation of~$\Gamma(\Theta_F) \backslash X_\bC(\Theta_F)$ such that
\begin{displaymath}
\xi(\kappa_1, \kappa_2)\Lambda_{\kappa_1}(x_1) = \Lambda_{\kappa_2}(x_2)
\end{displaymath}
for all simple inertial classes $x_1 = \JL_{G_2}(x_2)$ of endo-class~$\Theta_F$. 
By the results in section~\ref{lmodular} we see that for any prime number $\ell \not = p$ this permutation preserves the equivalence relation of having the same $\ell$-regular parts. 
We will prove that~$\xi(\kappa_1, \kappa_2)[\chi_1] = [\chi_1]$.

Because the parametric degree of simple inertial classes is preserved under the Jacquet--Langlands correspondence, one finds that~$\be[\chi_1] = \be[\chi_2]$ (since by the formulas in remark~\ref{invariants} the parametric degree of $\fs_{G_i}(\Theta_F, \Theta_E, [\chi_i])$ equals~$n/s[\chi_i]$). Hence~$\xi(\kappa_1, \kappa_2)$ preserves the size of Frobenius orbits.

Let~$a$ be some large integer ($a \geq 7$ will suffice) and write~$\kappa_i^*$ for the maximal $\beta$-extension in~$\GL_{an}(F)$ or~$\GL_{am}(D)$ compatible with~$\kappa_i$. 
Assume that~$\fs_i$ corresponds to the supercuspidal support~$\pi_i^{\otimes r_i}$, and let $\fs_{i, a}$ be the simple inertial class with supercuspidal support~$\pi_i^{\otimes a r_i}$. Then the~$\fs_i^*$ are Jacquet--Langlands transfers of each other, and we claim that that $\Lambda_{\kappa_i^*}(\fs_i^*)$ is the inflation~$[\chi_i^*]$ of~$\Lambda_{\kappa_i}(\fs_i)$. To see this, observe that $\pi_i$ is a supercuspidal representation of some~$\GL_{n_0}(F)$ or $\GL_{m_0}(D)$, and write~$\kappa_{i, *}$ for the $\beta$-extension in this group compatible with~$\kappa_i$. Then by  construction~$\Lambda_{\kappa_i}(\fs_i)$ is the inflation of~$\Lambda_{\kappa_{i,*}}(\pi_i)$. By transitivity, $\kappa_{i, *}$ and~$\kappa_i^*$ are compatible, hence~$\Lambda_{\kappa_i^*}(\fs_i^*)$ is the inflation of $\Lambda_{\kappa_{i,*}}(\pi_i)$, and the claim follows.

So $\xi(\kappa_1^*, \kappa_2^*)[\chi_1]^* = [\chi_2]^*$, and since the norm is surjective in finite extensions of finite fields it suffices to prove that  $\xi(\kappa_1^*, \kappa_2^*)[\chi_1]^* = [\chi_1]^*$. By~\cite[Lemma~8.5, Remark~8.7]{SecherreStevensJL} we can find a prime number~$\ell \neq p$ not dividing the order of~$\be[\chi_i]^\times$ and an~$\be$-regular character~$\beta$ of~$\be_{an/ \delta(\Theta_F)}^\times$ whose~$\ell$-regular part is~$\chi_1^*$. Since we know that $(\xi(\kappa_1^*, \kappa_2^*)[\chi_1^*])^{(\ell)} = (\xi(\kappa_1^*, \kappa_2^*)[\beta]])^{(\ell)}$, and~$\chi_2^*$ is also $\ell$-regular, it suffices to prove that $\xi(\kappa_1^*, \kappa_2^*)[\beta] = [\beta]$. 

By theorem~\ref{comparecharacters} we know that there exists some $\beta$-extension~$\varkappa$ in~$\GL_{an}(F)$ such that $\xi(\varkappa, \kappa_2^*)[\beta] = [\beta]$, hence there exists some character~$\delta$ of~$\be^\times$ such that $\xi(\kappa_1^*, \kappa_2^*)[\beta] = [\delta\beta]$ for \emph{every} $\be$-regular character~$\beta$ of~$\be^\times_{an/\delta(\Theta_F)}$, because~$\kappa_1^*$ and~$\varkappa$ are twists of each other. 
We will prove that~$\delta$ is trivial: this implies the theorem.

Fix some $\be$-regular character~$\alpha_+$ of~$\be_{n/\delta(\Theta_F)}^\times$. 
Again by~\cite[Lemma~8.5, Remark~8.7]{SecherreStevensJL} there exists some prime number $r \not = p$ not dividing the order of~$\be_{n/\delta(\Theta_F)}^\times = \be[\alpha_+]^\times$ and some $\be$-regular character $\beta_+$ of $\be_{an/\delta(\Theta_F)}^\times$ such that~$\alpha^*_+$ is the $r$-regular part of~$\beta_+$.

We know that $\xi(\kappa_1, \kappa_2)[\alpha_+] = [\alpha_+]$ by theorem~\ref{comparecharacters} and regularity of~$\alpha_+$. 
On the other hand, $\xi(\kappa_1^*, \kappa_2^*)[\alpha^*_+]$ is the $r$-regular part of $\xi(\kappa_1^*, \kappa_2^*)[\beta_+] = [\delta \beta_+]$, which is $[\delta^{(r)}\alpha^*_+]$.
Since $\xi(\kappa_1^*, \kappa_2^*)[\alpha^*_+] = (\xi(\kappa_1, \kappa_2)[\alpha_+])^*$ we find that $[\alpha^*_+] = [\delta^{(r)} \alpha^*_+]$.
Since~$\delta$ is a character of~$\be^\times$ and~$r$ does not divide the order of~$\be_{n/\delta(\Theta_F)}^\times$, we see that~$\delta^\pexp{r} = \delta$.
It follows that we can write $\delta =(\alpha^*_+)^{|\be|^{i}-1}$ for some integer $i \in \{0, \ldots, \frac{n}{\delta(\Theta_F)} - 1 \}$. 

Now we can take~$\alpha_+$ to be a generator of the character group of $\be_{n/\delta(\Theta_F)}^\times$, hence we can assume that~$\alpha_+$ has order $|\be|^{n/\delta(\Theta_F)} - 1$. 
But the equality $\delta =(\alpha^*_+)^{|\be|^{i}-1}$ implies that the order of~$\alpha^*_+$ divides $(|\be|^i - 1)(|\be| - 1)$. Since $|\be| \geq 2$ we have $|\be|^{n/\delta(\Theta_F)} - 1 > (|\be|^i - 1)(|\be| - 1)$, hence $i = 0$ and~$\delta$ is trivial.
\end{proof}

\bibliographystyle{amsalpha}
\bibliography{refpapers}

\end{document}